\documentclass[11pt]{amsart}

\usepackage{amssymb,amsfonts,amsthm,amsmath,mathrsfs,xspace,hyperref}
\usepackage{graphicx}
\usepackage[francais,english]{babel}
\usepackage[T1]{fontenc}
\usepackage[centering]{geometry}
\usepackage{csquotes}
\usepackage{enumerate}
\usepackage{enumitem}

  \newcommand{\R}{\ensuremath{\mathbb{R}}}%
  \newcommand{\Z}{\ensuremath{\mathbb{Z}}}%
  \newcommand{\N}{\ensuremath{\mathbb{N}}}%

    \newcommand{\A}{\ensuremath{\mathcal{A}}}%

        \newcommand{\M}{\ensuremath{\mathcal{M}}}%
                \renewcommand{\P}{\ensuremath{\mathcal{P}}}%
                \newcommand{\Sbb}{\ensuremath{\mathbb{S}}}%
                \renewcommand{\S}{\ensuremath{\mathcal{S}}}%
                \newcommand{\Ccal}{\ensuremath{\mathcal{C}}}%

        \renewcommand{\H}{\ensuremath{\mathcal{H}}}%
                \newcommand{\K}{\ensuremath{\mathcal{K}}}%

  \newcommand{\supp}{\ensuremath{\textrm{Supp}}}%
	\newcommand{\fix}{\ensuremath{\textrm{Fix}}}%
    \newcommand{\acts}{\ensuremath{\curvearrowright}}%
  \newcommand{\sub}{\ensuremath{\operatorname{Sub}}}%
  \newcommand{\stab}{\ensuremath{\operatorname{Stab}}}%
  \newcommand{\homeo}{\ensuremath{\operatorname{Homeo}}}%
    \newcommand{\aut}{\ensuremath{\operatorname{Aut}}}%
    \newcommand{\st}{\ensuremath{\operatorname{St}}}%
    \newcommand{\rist}{\ensuremath{\operatorname{RiSt}}}%
    \newcommand{\cay}{\ensuremath{\operatorname{Cay}}}%

    \newcommand{\PSL}{\ensuremath{\operatorname{PSL}}}%

	\newcommand{\PwG}{\mathrm{P_W}(G)}
	\newcommand{\autT}{\mathrm{Aut}(T)}
	
	\newcommand{\Cred}{C_{red}^\ast}
	\newcommand{\conj}{\mathcal{C}}
	\newcommand{\urs}{\ensuremath{\operatorname{URS}}}

\theoremstyle{definition}
  \newtheorem{defin}{Definition}[section]

\theoremstyle{plain}
  \newtheorem{thm}[defin]{Theorem}
  \newtheorem{main thm}{Theorem}
  \newtheorem{prop}[defin]{Proposition}
    \newtheorem{prop-def}[defin]{Proposition-Definition}

  \newtheorem{cor}[defin]{Corollary}
  \newtheorem{lemma}[defin]{Lemma}

\theoremstyle{remark}
  \newtheorem{remark}[defin]{Remark}
  \newtheorem{example}[defin]{Example}
	
\date{December 20, 2016}	
\title{Subgroup dynamics and $C^\ast$-simplicity of groups of homeomorphisms}
\author{Adrien Le Boudec}
\thanks{ALB is a F.R.S.-FNRS Postdoctoral Researcher}
\address{UCLouvain, IRMP,	Chemin du Cyclotron 2, 1348 Louvain-la-Neuve, Belgium}
\email{adrien.leboudec@uclouvain.be}

\author{Nicol\'as Matte Bon}
\thanks{NMB was partially supported by Projet ANR-14-CE25-0004 GAMME}
\address{Universit\'e Paris-Sud \& DMA, \'Ecole Normale Sup\'erieure, 45 Rue d'Ulm, 75005 Paris, France}
\email{nicolas.matte.bon@ens.fr}

\makeatletter
\@namedef{subjclassname@2010}{%
  \textup{2010} Mathematics Subject Classification}
\makeatother

\begin{document}
	
\maketitle

\begin{abstract}
We study the uniformly recurrent subgroups of groups acting by homeomorphisms on a topological space. We prove a general result relating uniformly recurrent subgroups to rigid stabilizers of the action, and deduce a $C^*$-simplicity criterion based on the non-amenability of rigid stabilizers. As an application, we show that Thompson's group $V$ is $C^\ast$-simple, as well as groups of piecewise projective homeomorphisms of the real line. This provides examples of finitely presented $C^\ast$-simple groups without free subgroups. We prove that a branch group is either amenable or $C^\ast$-simple. We also prove the converse of a result of Haagerup and Olesen: if Thompson's group $F$ is non-amenable, then Thompson's group $T$ must be $C^\ast$-simple. Our results further provide sufficient conditions on a group of homeomorphisms under which uniformly recurrent subgroups can be completely classified. This applies to Thompson's groups $F$, $T$ and $V$, for which we also deduce rigidity results for their minimal actions on compact spaces.

\end{abstract}


	
\section{Introduction}

Let $G$ be a second countable locally compact group. The set $\sub(G)$ of all closed subgroups of $G$ admits a natural topology, defined by Chabauty in \cite{Chabauty-topo}. This topology turns $\sub(G)$ into a compact metrizable space, on which $G$ acts continuously by conjugation. 

The study of $G$-invariant Borel probability measures on $\sub(G)$, named \textbf{invariant random subgroups} (IRS's) after \cite{A-G-V}, is a fast-developing topic \cite{A-G-V, 7s, Glas, B-D-L}. In this paper we are interested in their topological counterparts, called \textbf{uniformly recurrent subgroups} (URS's)~\cite{Glas-Wei}. A uniformly recurrent subgroup is a closed, minimal, $G$-invariant subset $\H\subset \sub(G)$. We will denote by $\operatorname{URS}(G)$ the set of uniformly recurrent subgroups of $G$. Every normal subgroup $N\unlhd G$ gives rise to a URS of $G$, namely the singleton $\{N\}$. More interesting examples arise from minimal actions on compact spaces: if $G$ acts minimally on a compact space $X$, then the closure of all point stabilizers in $\sub(G)$ contains a unique URS, called the \textbf{stabilizer URS} of $G\acts X$ \cite{Glas-Wei}.

When $G$ has only countably many subgroups (e.g.\ if $G$ is polycyclic), every IRS of $G$ is atomic, and every URS of $G$ is finite, as follows from a standard Baire argument. Leaving aside this specific situation, there are important families of groups for which a precise description of the space $\operatorname{IRS}(G)$ has been obtained. Considerably less is known about URS's. For example if $G$ is a lattice in a higher rank simple Lie group, the normal subgroup structure of $G$ is described by Margulis' Normal Subgroups Theorem. While Stuck--Zimmer's theorem \cite{Stu-Zim} generalizes Margulis' NST to IRS's, it is an open question whether a similar result holds for URS's of higher rank lattices, even for the particular case of $\operatorname{SL}(3,\Z)$ \cite[Problem 5.4]{Glas-Wei}. 

\subsection{Micro-supported actions}

In this paper we study the space $\urs(G)$ for countable groups $G$ admitting a faithful action $G \acts X$ on a topological space $X$ such that for every non-empty open set $U\subset X$, the \textbf{rigid stabilizer} $G_U$, i.e.\ the pointwise stabilizer of $X\setminus U$ in $G$, is non-trivial. Following \cite{Capr-proceed}, such an action will be called \textbf{micro-supported}. Note that this implies in particular that $X$ has no isolated points.

The class of groups admitting a micro-supported action includes Thompson's groups $F<T<V$ and many of their generalizations, groups of piecewise projective homeomorphisms of the real line \cite{Monod-Pw}, piecewise prescribed tree automorphism groups \cite{LB-c-etoile},  branch groups (viewed as groups of homeomorphisms of the boundary of the rooted tree) \cite{B-G-S-branch}, and topological full groups acting minimally on the Cantor set. These groups have uncountably many subgroups, and many examples in this class have a rich subgroup structure.



\medskip

Our first result shows that many algebraic or analytic properties of rigid stabilizers are inherited by the uniformly recurrent subgroups of $G$. In the following theorem and everywhere in the paper, a uniformly recurrent subgroup $\H \in \urs(G)$ is said to have a group property if every $H\in \H$ has the corresponding property.

\begin{thm}[see also Theorem \ref{thm-dycothomy-mostgeneral}] \label{MT: rigid urs}
Let $G$ be a countable group of homeomorphisms of a Hausdorff space $X$. Assume that for every non-empty open set $U\subset X$, the rigid stabilizer $G_U$ is non-amenable (respectively contains free subgroups, is not elementary amenable, is not virtually solvable, is not locally finite). Then every non-trivial uniformly recurrent subgroup of $G$ has the same property.
\end{thm}

This result has applications to the study of $C^\ast$-simplicity; see \textsection \ref{S: C^*-simple intro}.

\medskip

We obtain stronger conclusions on the uniformly recurrent subgroups of $G$ under additional assumptions on the action of $G$ on $X$. Recall that when $X$ is compact and $G\acts X$ is minimal, the closure of all point stabilizers in $\sub(G)$ contains a unique URS, called the \textbf{stabilizer URS} of $G\acts X$, and denoted $\S_G(X)$ \cite{Glas-Wei} (see Section \ref{sec-preliminaries} for details). The following result provides sufficient conditions under which $\S_G(X)$ turns out to be the unique URS of $G$, apart from the points $\{1\}$ and $\{G\}$ (hereafter denoted $1$ and $G$). We say that $G\acts X$ is an \textbf{extreme boundary action} if $X$ is compact and the action is minimal and extremely proximal (see \textsection\ref{subsec-notation} for the definition of an extremely proximal action).

\begin{thm}\label{thm-intro-3-urs}
Let $X$ be a compact Hausdorff space, and let $G$ be a countable group of homeomorphisms of $X$. Assume that the following conditions are satisfied: 
\begin{enumerate}[label=(\roman*)]
	\item $G\acts X$ is an extreme boundary action;
	\item there is a basis for the topology consisting of open sets $U \subset X$ such that the rigid stabilizer $G_U$ admits no non-trivial finite or abelian quotients;
	\item the point stabilizers for the action $G\acts X$ are maximal subgroups of $G$.
\end{enumerate}
Then the only uniformly recurrent subgroups of $G$ are $1$, $G$ and $\mathcal{S}_G(X)$.
\end{thm}

We actually prove a more general result, see Corollary \ref{cor-unique-urs}. Examples of groups to which this result applies are Thompson's groups $T$ and $V$, as well as examples in the family of groups acting on trees $G(F,F')$ (see \textsection \ref{subsec-intro-thompson} and \textsection \ref{subsec-intro-g(f,f')}). In particular this provides examples of finitely generated groups $G$  (with uncountably many subgroups) for which the space $\urs(G)$ is completely understood. In the case of Thompson's groups, we deduce from this lack of URS's rigidity results about their minimal actions on compact spaces (see \textsection\ref{subsec-intro-thompson}).

\subsection{Application to $C^\ast$-simplicity} \label{S: C^*-simple intro}

A group $G$ is said to be \textbf{$C^\ast$-simple} if its reduced $C^\ast$-algebra $\Cred(G)$ is simple. This property naturally arises in the study of unitary representations: $G$ is $C^\ast$-simple if and only if every unitary representation of $G$ that is weakly contained in the left-regular representation $\lambda_G$ is actually weakly equivalent to $\lambda_G$ \cite{dlH-survey}. Since amenability of a group $G$ is characterized by the fact that the trivial representation of $G$ is weakly contained in $\lambda_G$, a non-trivial amenable group is never $C^\ast$-simple.

The first historical $C^\ast$-simplicity result was Powers' proof that the reduced $C^\ast$-algebra of the free group $\mathbb{F}_2$ is simple \cite{Pow}. The methods employed by Powers have then been generalized in several different ways, and various classes of groups have been shown to be $C^\ast$-simple. We refer to \cite[Proposition 11]{dlH-survey} (see also the references given there), and to Corollary 12 therein for a list of important examples of groups to which these methods 
have been applied. 

Problems related to $C^\ast$-simplicity recently received new attention \cite{KK,BKKO,Raum,LB-c-etoile,Kenn,Ha-newlook}. A characterization of $C^\ast$-simplicity in terms of boundary actions was obtained by Kalantar and Kennedy: a countable group $G$ is $C^\ast$-simple if and only if $G$ acts topologically freely on its Furstenberg boundary; equivalently, $G$ admits some topologically free boundary action \cite{KK} (we recall the terminology in \textsection\ref{subsec-notation}). By developing a systematic approach based on this criterion, Breuillard--Kalantar--Kennedy--Ozawa provided new proofs of $C^\ast$-simplicity for many classes of groups~\cite{BKKO}. Moreover, they showed that the uniqueness of the trace on $\Cred(G)$ actually characterizes the groups $G$ having trivial amenable radical~\cite{BKKO}. While $C^\ast$-simplicity also implies triviality of the amenable radical \cite{dlH-survey}, counter-examples to the converse implication have been given in \cite{LB-c-etoile}.

Relying on the aforementioned boundary criterion from \cite{KK}, Kennedy subsequently showed that a countable group $G$ is $C^\ast$-simple if and only if $G$ admits no non-trivial amenable URS \cite{Kenn}. In view of this result, Theorem \ref{MT: rigid urs} has the following consequence.

\begin{cor}\label{cor-intro-csimple}
Let $X$ be a Hausdorff space, and let $G$ be a countable group of homeomorphisms of $X$. Assume that for every non-empty open set $U\subset X$, the rigid stabilizer $G_U$ is non-amenable. Then $G$ is $C^\ast$-simple.
\end{cor}

We use this criterion to prove the $C^\ast$-simplicity of several classes of groups described later in this introduction.

\medskip

Consider now a countable group $G$ and a given boundary action $G\acts X$. By \cite{KK} if $G\acts X$ is topologically free then $G$ is $C^\ast$ simple. The converse does not hold: the fact that $G\acts X$ is not topologically free surely does not rule out the existence of another topologically free boundary action. Nevertheless, one may still ask if the $C^\ast$-simplicity of $G$ can be characterized in terms of the given action $G\acts X$ only. This is useful in practice since it often happens that groups come equipped with an explicitly given boundary action, and other boundary actions may be difficult to concretely identify. It is proven in \cite{BKKO} that if $G\acts X$ is not topologically free and has amenable stabilizers, then $G$ is not $C\ast$-simple. If however stabilizers are non-amenable, nothing can be concluded on the $C^\ast$-simplicity of $G$ (see Example \ref{exe-stab-nonamenable}).

We show that under the additional assumption that $G\acts X$ is an \emph{extreme} boundary action, the $C^\ast$-simplicity of $G$ is completely characterized by the stabilizers of $G\acts X$.
  
\begin{thm} \label{MT: extreme}
Let $G$ be a countable group, and $G\acts X$ a faithful extreme boundary action. Then $G$ is $C^\ast$-simple if and only if one of the following possibilities holds:
\begin{enumerate}[label=(\roman*)]
\item the action $G\acts X$ is topologically free;
\item the point stabilizers of the action $G\acts X$ are non-amenable.
\end{enumerate}
\end{thm}

In fact in Theorem \ref{thm-AG-extreme} we elucidate the precise relation between the stabilizers of any faithful extreme boundary action $G\acts X$, and the stabilizers of the action of $G$ on its Furstenberg boundary.
Examples of groups that admit a natural extreme boundary action are Thompson's groups, as well as any group admitting an action on a tree which is minimal and of general type (see \textsection \ref{subsec-trees}).

\subsection{Thompson's groups} \label{subsec-intro-thompson}

Recall that Thompson's group $F$ is the group of orientation preserving homeomorphisms of the unit interval which are piecewise linear, with only finitely many breakpoints, all at dyadic rationals, and slopes in $2^{\mathbb{Z}}$. Thompson's group $T$ admits a similar description as group of homeomorphisms of the circle, and $V$ is a group of homeomorphisms of the Cantor set. We refer the reader to the notes \cite{C-F-P} for an introduction to these groups. 

Corollary \ref{cor-intro-csimple} admits the following consequence.

\begin{thm}\label{thm-intro-V}
Thompson's group $V$ is $C^\ast$-simple.
\end{thm}

Similar arguments apply to various classes of \enquote{Thompson-like} groups, including the higher dimensional generalizations $nV$ introduced by Brin \cite{Brin-nV}, as well as the groups $V_G$ associated to a self-similar group $G$ constructed by Nekrashevych \cite{Nek-cuntz,Nek-fp}. See Theorem \ref{thm-V-and-relatives}.

\medskip

Recall that it is a long-standing open problem to determine whether $F$ is amenable. Haagerup and Olesen discovered in \cite{H-O} a connection between this problem and the $C^\ast$-simplicity of $T$. They proved that if the group $F$ is amenable, then the group $T$ cannot be $C^\ast$-simple. Whether the converse statement is true has been considered by Bleak and Juschenko in \cite{B-J}, and more recently in \cite{Bleak-normalish}. Breuillard--Kalantar--Kennedy--Ozawa obtained a partial result in this direction, proving that if $T$ is not $C^*$-simple, then $F$ is not $C^*$-simple either \cite{BKKO}. 

In this paper we prove the converse of the Haagerup--Olesen result, i.e.\ that the non $C^\ast$-simplicity of $T$ would imply amenability for $F$. More generally, we show that the existence of at least one overgroup of $F$ inside $\homeo(\Sbb^1)$ which is not $C^*$-simple, would imply that $F$ is amenable (see Theorem \ref{thm-F-overgroups-circle}). This applies in particular to the group $F$ itself, so that the non-amenability of $F$ would automatically imply its $C^\ast$-simplicity.

\begin{thm}[see also Theorem \ref{thm-T-set-points-stab}]
The following statements about Thompson's groups $F$ and $T$ are equivalent:
\begin{enumerate}[label=(\roman*)]
\item The group $F$ is non-amenable;
\item The group $F$ is $C^*$-simple;
\item The group $T$  is $C^*$-simple.
\end{enumerate}
\end{thm}

We point out that, while we prove these properties to be equivalent, we do not elucidate whether these are true or false.

\medskip

We also obtain a complete classification of the URS's of Thompson's groups. Recall that R.\ Thompson proved that the groups $[F,F]$, $T$ and $V$ are simple, and that the normal subgroups of $F$ are precisely the subgroups containing $[F,F]$ \cite{C-F-P}. In spite $T$ and $V$ do not have normal subgroups, they do admit non-trivial URS's coming from their action respectively on the circle and the Cantor set. We prove that these are the only ones, and that $F$ admits no URS other than its normal subgroups.

\begin{thm}[Classification of the URS's of Thompson's groups] \label{thm-intro-URS-Thompson}
$ $
\begin{enumerate}[label=(\roman*)]
\item The only URS's of Thompson's group $F$ are the normal subgroups. The derived subgroup $[F,F]$ has no uniformly recurrent subgroups other than $1$ and $[F,F]$.
\item The URS's of Thompson's group $T$ are $1, T$ and the stabilizer URS associated to its action on the circle.
\item The URS's of Thompson's group $V$ are $1, V$ and the stabilizer URS associated to its action on the Cantor set.
\end{enumerate}
\end{thm}

Theorem \ref{thm-intro-URS-Thompson} can be compared with a result of Dudko and Medynets \cite{Dud-Med}, stating that the groups $F$, $T$ and $V$ essentially do not admit invariant random subgroups. Note that the URS's associated to $T\acts \Sbb^1$ and $V\acts \Ccal$ do not carry any invariant probability measure.

\medskip

This lack or uniformly recurrent subgroup implies a rigidity result on the possible minimal actions on compact spaces of the groups $F, T$ and $V$. Recall that if $X,Y$ are two $G$-spaces, the action $G\acts X$ is said to  \textbf{factor onto}  $G\acts Y$ if there exists a continuous $G$-equivariant map from $X$ onto $Y$.

\begin{thm}[Rigidity of minimal actions of Thompson's groups on compact spaces] \label{thm-intro-actions-Thompson}
$ $
\begin{enumerate}[label=(\roman*)]
\item Every faithful, minimal action of $F$ on a compact space is topologically free.
\item Every non-trivial minimal action of $T$ on a compact space is either topologically free or factors onto the standard action on the circle. 
\item Every non-trivial minimal action of $V$ on a compact space is either topologically free or factors onto the standard action on the Cantor set. 
\end{enumerate}
\end{thm}

Part (ii) of Theorem \ref{thm-intro-actions-Thompson} may be compared to a result of Ghys and Sergiescu, stating that every non-trivial action of $T$ on the circle by $C^2$-diffeomorphisms is semi-conjugate to the standard one \cite[Th\'eor\`eme K]{Ghy-Se}. Whether this rigidity holds for actions by homeomorphisms does not seem to
have been adressed in the literature (see the remark following Theorem 3.14 in \cite{Mann-circle}). In \textsection\ref{S: T circle} we show how Theorem \ref{thm-intro-actions-Thompson} can be used to prove this. We learned from \'E.\ Ghys that this can also be proved using bounded cohomology.

\begin{cor} \label{cor-Ghys-Serg-homeos}
Every non-trivial action by homeomorphisms of Thompson's group $T$ on the circle is semi-conjugate to the standard one.
\end{cor}

\subsection{Piecewise projective homeomorphisms of the real line}

Following \cite{Monod-Pw}, if $A$ is a subring of $\mathbb{R}$, we denote by $G(A)$ the group of homeomorphisms of the projective line $\mathbb{P}^1(\mathbb{R})$ which are piecewise $\mathrm{PSL}(2,A)$, each piece being a closed interval, with breakpoints in the set of ends of hyperbolic elements of $\mathrm{PSL}(2,A)$. Let also $H(A)$ be the stabilizer of the point $\infty$ in $G(A)$. By work of Monod, when $A$ is a dense subring of $\mathbb{R}$, the group $H(A)$ is a counter-example to the so-called von Neumann-Day problem: $H(A)$ is non-amenable and yet does not contain any non-abelian free subgroups \cite{Monod-Pw}.

Lodha and Moore \cite{Lod-Moo} have exhibited a non-amenable $3$-generated group $G_0 \le H(\R)$, and proved that $G_0$ is finitely presented. The definition of the group $G_0$ is recalled in \textsection\ref{subsec-projective}.

\medskip

de la Harpe asked in \cite{dlH-survey} whether there exist countable $C^\ast$-simple groups with no free subgroups. This was answered in the positive by Olshanskii and Osin in \cite{O-O}. Their examples are direct limits of relatively hyperbolic groups with surjective homomorphisms $G_n \twoheadrightarrow G_{n+1}$ \cite{O-O}. In particular, these groups are never finitely presented. More examples were given in \cite[Theorem 1.6]{KK} and \cite[Corollary 6.12]{BKKO}, where $C^\ast$-simplicity of the Tarski monster groups constructed by Olshanskii in \cite{Ol-Tarski1, Ol-Tarski2} has been obtained.

Here we show that groups of piecewise projective homeomorphisms provide new examples of $C^\ast$-simple groups without free subgroups. In particular, the $C^\ast$-simplicity of the group $G_0$ shows that the question of de la Harpe also has a positive answer in the realm of finitely presented groups.

\begin{thm}
Let $A$ be a countable dense subring of $\mathbb{R}$. Then both $H(A)$ and $G(A)$ are $C^\ast$-simple. Moreover the Lodha-Moore group $G_0$ is $C^\ast$-simple. 
\end{thm}

\subsection{Groups acting on trees with almost prescribed local action} \label{subsec-intro-g(f,f')}

In this paragraph we consider the family of groups $G(F,F')$ which have been proved to be non $C^\ast$-simple in \cite{LB-c-etoile} and yet do not have non-trivial amenable normal subgroups. We first briefly recall their definition.

Let $\Omega$ be a set, and $F \leq F' \leq \mathrm{Sym}(\Omega)$ two permutation groups on $\Omega$ such that $F$ acts freely transitively on $\Omega$. Here we assume $\Omega$ to be finite for simplicity, but in \textsection\ref{subsec-G(F,F')} we will allow $\Omega$ countable. We let $T$ be a $\left|\Omega\right|$-regular tree, whose edges are coloured by the elements of $\Omega$, so that neighbouring edges have different colors. The group $G(F,F')$ is by definition the group of automorphisms of $T$ whose local action is prescribed by the permutation group $F'$ for all vertices, and by the permutation group $F$ for all but finitely many vertices. We refer to \textsection \ref{subsec-G(F,F')} for a formal definition, and to \cite{LB-ae} for properties of these groups. The subgroup of index two of $G(F,F')$ that preserves the types of vertices of $T$ will be denoted $G(F,F')^\ast$.

The groups $G(F,F')$ have no non-trivial amenable normal subgroup, but do have non-trivial amenable URS's \cite{LB-c-etoile}. Indeed, the action of the group $G(F,F')$ on the tree $T$ extends to a minimal action by homeomorphisms on the set of ends $\partial T$, which has amenable stabilizers and which is not topologically free \cite{LB-c-etoile}. In terms of URS's, this exactly means that the stabilizer URS associated to the action $G(F,F') \acts \partial T$ is amenable and non-trivial. A natural problem then arises, which is to classify all amenable URS's of the groups $G(F,F')$. The following result provides, beyond the amenable case, a complete classification of all URS's of these groups under appropriate assumptions of the permutation groups.

\begin{thm}[see also Proposition \ref{prop--urs-tree-G(F,F')} and Theorem \ref{thm-all-urs-g(f,f')}]
Let $F \leq F' \leq \mathrm{Sym}(\Omega)$ such that $F$ acts freely transitively on $\Omega$, $F'$ acts $2$-transitively on $\Omega$, and point stabilizers in $F'$ are perfect. Write $G = G(F,F')^\ast$. Then the following hold:
\begin{enumerate}[label=(\roman*)]
	\item $G$ admits exactly three URS's, namely $1$, $\mathcal{S}_G(\partial T)$ and $G$; where $\mathcal{S}_G(\partial T)$ is the stabilizer URS associated to the action $G \acts \partial T$.
	\item $\mathcal{S}_G(\partial T)$ is the unique non-trivial amenable URS of $G$, and we have $\mathcal{S}_G(\partial T) = \left\{G_{\xi,0} \, : \, \xi \in \partial T\right\}$, where $G_{\xi,0}$ is the set of elliptic elements of $G$ that fix $\xi$, and is an infinite locally finite group.
\end{enumerate} 
\end{thm}

In particular this provides examples of finitely generated groups with trivial amenable radical and exactly one non-trivial amenable URS.

\medskip

Although the boundary $\partial T$ is \textit{not} the Furstenberg boundary of $G(F,F')$, we are able to precisely identify the point stabilizers in $G(F,F')$ for the action on the Furstenberg boundary. In particular we characterize the elements $g \in G(F,F')$ that have fixed points in the Furstenberg boundary of $G(F,F')$. See Corollary \ref{cor-fix-G(F,F')-furst}. 

\subsection{$C^\ast$-simplicity for branch groups}

Branch groups are a class of groups acting on rooted trees that naturally appears in the classification of just-infinite groups. The class of branch groups contains many instances with interesting properties such as Grigorchuk groups $G_\omega$ of intermediate growth \cite{Grigorchuk:grigorchukgroups}, or Gupta-Sidki torsion groups \cite{Gupta-Sidki}. We refer the reader to \cite{B-G-S-branch} for a survey on branch groups. The study of amenability within the class of branch groups has been actively investigated. Several examples of well-studied branch groups are amenable (e.g.\ the groups mentioned above \cite{Grigorchuk:grigorchukgroups, B-K-N}), but Sidki and Wilson constructed finitely generated branch groups containing free subgroups \cite{S-W-branch}, therefore the branch property does not imply amenability. 

We show that the following sharp dichotomy holds in the class of branch groups.

\begin{thm}\label{thm-intro-branch}
A countable branch group is either amenable or $C^\ast$-simple.
\end{thm}

This will follow from a more general statement saying that many properties of a branch group are inherited by its uniformly recurrent subgroups, which will follow from Theorem \ref{MT: rigid urs} applied to the action on the boundary of the rooted tree, see \textsection\ref{subsec-branch}.

\subsection{$C^\ast$-simplicity for topological full groups}

Let $\Gamma$ be a group acting by homeomorphism on the Cantor set $X$. The \textbf{topological full group} of $\Gamma\acts X$ is the group $[[\Gamma]]$ of all homeomorphisms of $X$ that locally coincide with an element of $\Gamma$. This notion was first introduced in \cite{G-P-S}, in the case $\Gamma=\Z$, in connection with the study of orbit equivalence of Cantor minimal systems. See \cite{Matui:survey} for a recent survey (in the more general setting of \'etale groupoids). One feature of this construction from the group-theoretical point of view is that it provides many new examples of finitely generated, infinite simple groups \cite{Matui:simple, Nek:simple}.
Juschenko and Monod proved that the topological full group of any Cantor minimal $\Z$-system is amenable \cite{Ju-Mo}, providing  the first examples of finitely generated, infinite simple groups that are amenable. This result motivated the study of analytic properties of  topological full groups. Amenability of other families of topological full groups was established in \cite{J-N-S, J-M-M-S}. Recently Nekrashevych constructed \'{e}tale groupoids whose topological full groups have intermediate growth \cite{Nek:intgrowth}, giving the first examples of simple groups with this property. Other analytic properties of topological full groups that have been studied include the Haagerup property \cite{Matui:finitetype}, and the Liouville property~\cite{MB:Liouville}.

All these properties are either strong or weak forms of amenability. Here we go in the opposite direction and consider the question of determining when the topological full group of a group action is $C^\ast$-simple.

\begin{thm}
Let $\Gamma$ be a non-amenable group acting freely and minimally on the Cantor set. Then the topological full group $[[\Gamma]]$ is $C^\ast$-simple.
\end{thm}

\subsection*{Organization}

The rest of this article contains three additional sections. In Section \ref{sec-preliminaries} we set some notation and give preliminaries about uniformly recurrent subgroups. In particular we explain the existence, for every countable group $G$, of an amenable uniformly recurrent subgroup $\mathcal{A}_G$ that is larger (with respect to a natural order on the set $\urs(G)$) than any other amenable URS. We also explain the connection between $\mathcal{A}_G$ and $C^\ast$-simplicity of $G$.

The study of uniformly recurrent subgroups in groups admitting a micro-supported action is developed in Section \ref{sec-main-thms}. A key ingredient used throughout this section is Proposition \ref{prop-technique}. We point out that the reader interested in the proof of Theorem \ref{MT: rigid urs} and its applications to $C^\ast$-simplicity may skip \textsection\ref{SS: EP}, and only needs a simplified version of Proposition \ref{prop-technique} (see the remark before its proof). Theorem~\ref{thm-intro-3-urs} and Theorem~\ref{MT: extreme} are proved at the end of Section \ref{sec-main-thms}.

Section \ref{sec-applications} concerns the applications of the results of the previous sections to various classes of groups, and contains the proofs of all the other results mentioned in the introduction. Each of its subsections can be read independently from the others, after reading Sections \ref{sec-preliminaries} and \ref{sec-main-thms}.

\subsection*{Acknowledgements} We are grateful to Pierre-Emmanuel Caprace, Anna Erschler and Pierre de la Harpe for their useful comments. We thank Bertrand Deroin for pointing out Th\'eor\`eme K in \cite{Ghy-Se}, and \'Etienne Ghys for useful explanations about it.

Part of this work was developed during the program \textit{Measured group theory} in Vienna in 2016. We would like to thank the organizers, as well as the \textit{Erwin Schr\"{o}dinger International Institute for Mathematics and Physics} for its hospitality.
\medskip
	
\section{Preliminaries} \label{sec-preliminaries}

\subsection{Notation and terminology} \label{subsec-notation}

Let $X$ be a topological Hausdorff space, and $G$ a countable group acting faithfully by homeomorphisms on $X$. For $x \in X$, we will denote by $G_x$ the stabilizer of $x$ in $G$, and by $G_x^0$ the (normal) subgroup of $G_x$ consisting of elements $g \in G_x$ such that $g$ fixes pointwise a neighbourhood of $x$.

For $g \in G$, we will denote by $\fix(g)$ the set of fixed points of $g$, and by $\supp(g)$ the support of $g$, i.e.\ the closure of the set of $x \in X$ that are moved by $g$. If $H$ is a subgroup of $G$ and $Y \subset X$, we will say that $H$ is supported in $Y$ if $\supp(h) \subset Y$ for every $h \in H$. Given a subset $U \subset X$, we will denote by $G_U$ the \textbf{rigid stabilizer} of $U$, which is defined as the subgroup of $G$ consisting of elements that fix pointwise the complement of $U$.

The action of $G$ on $X$ is:

\begin{itemize}
 \item \textbf{micro-supported} if $G_U$ is non-trivial for every non-empty open $U$.
 \item \textbf{topologically free} if $\fix(g)$ has empty interior for every non-trivial $g \in G$. When $X$ is a Baire space (e.g.\ $X$ is compact), this is equivalent to saying that the set of points $x \in X$ such that $G_x$ is trivial is a $G_\delta$-dense subset of $X$.
	\item \textbf{minimal} if every orbit is dense in $X$, or equivalently if $X$ is the only non-empty closed invariant subset.
	\item \textbf{strongly proximal} if the closure of the orbit of every probability measure on $X$ contains a Dirac mass.
	\item a \textbf{boundary action} if $X$ is compact, and the action of $G$ is minimal and strongly proximal.
	\item \textbf{extremely proximal} if every closed $C \neq X$ is compressible, where $Y \subset X$ is \textbf{compressible} if there exists a point $x \in X$ such that for every open $U \subset X$ containing $x$, there exists $g \in G$ such that $g(Y) \subset U$.
	\item an \textbf{extreme boundary action} if $X$ is compact, and the action of $G$ is minimal and extremely proximal. Note that an extreme boundary action is a boundary action \cite[Theorem 2.3]{Glas-top-dyn-group}. 
\end{itemize}  

\subsection{The Chabauty space}

If $G$ is a countable group, we denote by $\sub(G)$ the space of all subgroups of $G$. When viewed as a subset of $\{0,1\}^G$, the space $\sub(G)$ is closed for the product topology. The topology induced on $\sub(G)$ by the product topology is  called the \textbf{Chabauty topology}, and makes $\sub(G)$ a compact space. A sequence $(H_n)$ converges to $H$ in $\sub(G)$ if and only if every element of $H$ eventually belongs to $H_n$, and $H$ contains $\cap_k H_{n_k}$ for every subsequence $(n_k)$. Note in particular that a sequence $(H_n)$ converges to the trivial subgroup if and only if $1$ is the only element of $G$ that belongs to $H_n$ for infinitely many $n$.

The action of the group $G$ on itself by conjugation naturally extends to an action of $G$ on $\sub(G)$ by homeomorphisms. For $H \in \sub(G)$, we will denote by $\conj(H)$ the conjugacy class of $H$ in $G$, i.e.\ $\conj(H)$ is the $G$-orbit of $H$ in $\sub(G)$.

\bigskip

The following easy lemma will be used in Section \ref{sec-main-thms}. 

\begin{lemma}\label{L: accumulation}
Let $G$ be a countable group, and $H \in \sub(G)$. The following are equivalent:
\begin{enumerate}[label=(\roman*)]
\item The closure of $\mathcal{C}(H)$ in $\sub(G)$ does not contain the trivial subgroup;
\item There exists a finite subset $P\subset G\setminus\{1\}$ all of whose conjugates intersect $H$; 
\item There exists a finite subset $P\subset G\setminus\{1\}$ such that for every $K \in \mathcal{C}(H)$, all the conjugates of $P$ intersect $K$.
\end{enumerate}
\end{lemma}

\begin{proof}
By definition of the Chabauty topology, a basis of neighbourhoods of the trivial subgroup is given by the sets
\[U_P=\{L \in \sub(G)\: :\: L \cap P=\varnothing\},\]
where $P$ ranges over finite subsets of $G\setminus \{1\}$. The equivalence between (i) and (ii) follows immediately, and the equivalence between (ii) and (iii) is clear.
\end{proof}
If $G$ acts by homeomorphisms on a space $X$, we may consider the stabilizer map $\mathrm{Stab}: X \rightarrow \sub(G)$, defined by $x \mapsto G_x$. This map need not be continuous in general. The following lemma, which characterizes its continuity points, has already been proved in \cite[Lemma 5.4]{Vo}. We include a proof for completeness.

\begin{lemma} \label{lem-cont-stab}
Let $G$ be a countable group acting by homeomorphisms on a Hausdorff space $X$. Then the map $\mathrm{Stab}: X \rightarrow \sub(G)$ is continuous at $x \in X$ if and only if $G_x = G_x^0$, or equivalently $\fix(g)$ contains a neighbourhood of $x$ for every $g \in G_x$.
\end{lemma}

\begin{proof}
The map $y\mapsto G_y\in \sub(G)\subset\{0,1\}^G$ is continuous if and only if its post-compositions with all the projections of the product space $\{0,1\}^G$ onto its factors are continuous. For a fixed $g\in G$, the post-composition with the corresponding projection is given by $y\mapsto 1_{gy=y}$. Now if $gx\neq x$, this map is obviously continuous at $x$, while if $gx=x$ this map is continuous if and only if $g\in G_x^0$. Hence it is continuous for every $g\in G$ if and only if $G_x=G^0_x$. \qedhere

%
\end{proof}

\begin{defin} If $G$ is a group acting by homeomorphisms on a Hausdorff space $X$, we will denote by $X_0\subset X$ the domain of continuity of the map $\mathrm{Stab}$.
\end{defin}

\begin{prop} \label{prop-dense-Gdelta}
If $X$ is a Baire space, $X_0$ is a dense $G_\delta$ subset of $X$.
\end{prop}

\begin{proof}
By Lemma \ref{lem-cont-stab}, $X_0$ is exactly the complement of $\cup_{g\in G} \partial \fix(g)$. Since each $\partial \fix(g)$ has empty interior (because $\fix(g)$ is always closed), the statement follows. \qedhere
\end{proof}

\subsection{Uniformly recurrent subgroups}

The notion of uniformly recurrent subgroup was introduced and investigated in \cite{Glas-Wei}. A \textbf{uniformly recurrent subgroup} (URS for short) is a closed minimal $G$-invariant subset of $\sub(G)$. The set of all URS's of $G$ will be denoted $\urs(G)$.

Examples of URS's are normal subgroups, and more generally subgroups with a finite conjugacy class. For simplicity the URS associated to a normal subgroup $N$ will still be denoted $N$ rather than $\left\{N\right\}$. The \textbf{trivial URS}, denoted $1$, is the URS that contains only the trivial subgroup. If ($\mathcal{Q})$ is a property of groups, we say that $\mathcal{H} \in \urs(G)$ has ($\mathcal{Q})$ if every $H \in \mathcal{H}$ has ($\mathcal{Q})$.


\bigskip


Following Glasner and Weiss, we recall the construction of a URS starting from a minimal action on a compact space \cite{Glas-Wei}.  

\begin{prop} \label{prop-min-action-urs}
Assume that $G$ acts minimally on a compact space $X$, and denote $X_0\subset X$ the locus of continuity points of  $\stab$. Let $\mathcal{S}_G(X)$ be the closure in $\sub(G)$ of the set $\{G_x, x\in X_0\}$. Then $\mathcal{S}_G(X)$ is a URS of $G$.

Moreover $\mathcal{S}_G(X)$ is the only closed minimal $G$-invariant subset of the closure in $\sub(G)$ of the collection of $G_{x}$, $x \in X$. 
\end{prop}

\begin{proof}
See the proof of Proposition 1.2 in \cite{Glas-Wei}. Note that the assumption made throughout \cite{Glas-Wei} that $X$ is metrizable is not needed here: the only argument used in the proof is the density of $X_0$ in $X$, which follows from Proposition \ref{prop-dense-Gdelta}. 
\end{proof}


\begin{defin}
$\mathcal{S}_G(X)$ will be called the \textbf{stabilizer URS} of the action $G\acts X$. 
\end{defin}

The following proposition shows that this notion is consistent with the terminology of \enquote{topologically free action}.

\begin{prop} \label{prop-trivial-URS-top-free}
Let $G\acts X$ be a minimal action on a compact space. Then
$\mathcal{S}_G(X)$ is the trivial URS if and only if the action of $G$ on $X$ is topologically free.
\end{prop}

\begin{proof}
If $\mathcal{S}_G(X)$ is trivial then one has $G_{x_0} =1$ for every $x_0 \in X_0$. Since $X_0$ is dense in $X$ by Proposition \ref{prop-dense-Gdelta}, this implies that the action of $G$ on $X$ is topologically free. Conversely if $G\acts X$ is topologically free, then in particular there exists $x\in X$ such that $G_x=1$. The conclusion then follows from the last affirmation in Proposition \ref{prop-min-action-urs}. \qedhere
\end{proof}

We will need the following observation.

\begin{lemma}\label{L: upper semicontinuous}
Let $G\acts X$ be a minimal action on a compact space, and let $\S_G(X)\in \urs(G)$ be the stabilizer URS. Then 
\begin{enumerate}[label=(\roman*)]
\item for every $x\in X$, there exists $H\in \S_G(X)$ (not necessarily unique) such that $H\le G_x$;
\item for every $H\in \S_G(X)$ there exists $x\in X$ such that $G_x^0\le H$.

\end{enumerate}
\end{lemma}

\begin{proof}
To prove (i), let $x \in X$. By Proposition \ref{prop-dense-Gdelta}, we may find a net $(x_i)$ of points of $X_0$ that converges to $x$. Up to taking a sub-net, we may assume that $(G_{x_i})$ converges to a limit $H\in \sub(G)$. Note that $H \in \S_G(X)$ by Proposition \ref{prop-min-action-urs}. Moreover every $g\in H$ eventually belongs to $G_{x_i}$, so $g$ also belongs to $G_x$ since $x_i\to x$. Hence $H\le G_x$.

To prove (ii), let $H\in \S_G(X)$. By Proposition \ref{prop-min-action-urs}, we can find a net $(x_i)$ of points of $X_0$ such that $G_{x_i}$ converges to $H$. Up to taking a sub-net we may assume that $(x_i)$ also converges to a point $x\in X$. Then every element of $G^0_x$ eventually belongs to $G_{x_i}$. It follows that $G^0_x\le H$.
 \qedhere
\end{proof}

Under an additional assumption on the action $G\acts X$, it is possible to give a more explicit description of the stabilizer URS $\S_G(X)$.

\begin{defin}\label{D: Hausdorff germs}
Let $G$ be a countable group acting on a topological space $X$. The action $G\acts X$ is said to have \textbf{Hausdorff germs} if for every $x\in X$ and every $g \in G_x$, either
\begin{enumerate}[label=(\roman*)]
\item  $\fix(g)$ contains a neighbourhood of $x$; or
\item the set of interior points of $\fix(g)$ does not accumulate on $x$. 
\end{enumerate}
\end{defin}

The terminology is motivated by the fact that these conditions exactly characterize the actions $G\acts X$ whose \emph{groupoid of germs} is Hausdorff.

\begin{prop} \label{prop-cont-stab0}
Let $G\acts X$ with Hausdorff germs. Then the map $\mathrm{Stab}^0: X \rightarrow \sub(G)$, $x \mapsto G_x^0$, is continuous. If moreover $X$ is compact and $G\acts X$ is minimal, then $\mathcal{S}_G(X) = \{G_x^0\, \mid\, x\in X\}$.
\end{prop}

\begin{proof}
Fix $x \in X$. We let $P$ be a finite subset of $G$, and we shall prove that the set of $y \in X$ such that $G_y^0 \cap P = G_x^0 \cap P$ contains a neighbourhood of $x$.

By assumption we may partition the subset $P = \left\{g_1,\ldots,g_i,g_{i+1},\ldots,g_j\right\}$, such that $g_1,\ldots,g_i$ belong to $G_x^0$; and $g_{i+1},\ldots,g_j$ are such that the second condition of Definition \ref{D: Hausdorff germs} is satisfied.

For every $k \in \left\{1,\ldots,j\right\}$, we choose a neighbourhood $U_k$ of $x$ such that:
\begin{itemize}
\item $U_k \subset \fix(g_k)$ if $k \in \left\{1,\ldots,i\right\}$;
\item $U_k \cap \fix(g_k)$ has empty interior if $k \in \left\{i+1,\ldots,j\right\}$;
\end{itemize}
and we consider $U = \cap U_k$. The set $U$ is clearly a neighbourhood of $x$, and by construction for every $y \in U$, the element $g_k$ belongs to $G_y^0$ if and only if $k \in \left\{1,\ldots,i\right\}$. This shows that the map $\mathrm{Stab}^0$ is continuous at $x$.

Since $X$ is compact, the set of subgroups $G_x^0$, $x \in X$, is therefore a compact subset of $\sub(G)$, and it is clearly invariant. Moreover it is minimal as soon as the action of $G$ on $X$ is minimal, and it contains $\mathcal{S}_G(X)$ by Lemma \ref{lem-cont-stab}. Therefore it must coincide with $\mathcal{S}_G(X)$. \qedhere
\end{proof}

\subsection{The largest amenable uniformly recurrent subgroup} \label{subsec-largest-urs}

Given a countable group $G$, the set $\urs(G)$ can be naturally endowed with a partial order, as we shall now explain. We are grateful to P-E.\ Caprace for suggesting to use this notion to formalize our results.

In this paragraph we will use letters $A,B$ for arbitrary subsets of $\sub(G)$, and $\mathcal H, \mathcal K$ for uniformly recurrent subgroups of $G$.

\begin{defin}
For $A, B \subset \sub(G)$, we  write $A \preccurlyeq B$ if there exist $H \in A$ and $K \in B$ such that $H$ is contained in $K$.
\end{defin}

The relation $\preccurlyeq$ is neither transitive nor antisymmetric on the set of all subsets of $\sub(G)$. However, we shall prove below that, when restricted to to $\urs(G)$, it becomes a partial order.

For every subset $A \subset \sub(G)$, we will denote by $\mathcal{UE}(A)$ the set of $H \in \sub(G)$ such that there exists $K\in A$ with $K \leq H$. The set $\mathcal{UE}(A)$ will be called the \textbf{upper envelope} of $A$. Similarly we will denote by $\mathcal{LE}(A)$ the set of $H \in \sub(G)$ such that there exists $K\in A$ with $H \leq K$. The set $\mathcal{UE}(A)$ will be called the \textbf{lower envelope} of $A$. Note that if $A,B \subset \sub(G)$, we have $A \preccurlyeq B$ if and only if $B$ intersects $\mathcal{UE}(A)$, if and only if $A$ intersects $\mathcal{LE}(B)$.

\medskip

The proof of the following lemma is an easy consequence of the definition of the Chabauty topology, and we omit it.


\begin{lemma} \label{lem-env-closed}
If $A$ is a closed $G$-invariant subset of $\sub(G)$, then $\mathcal{UE}(A)$ and $\mathcal{LE}(A)$ are also closed and $G$-invariant.
\end{lemma}

We will need the following lemma.

\begin{lemma} \label{lem-prec-unique-min}
Assume that $A,B \subset \sub(G)$ are closed $G$-invariant subsets of $\sub(G)$, and that $B$ admits a unique minimal $G$-invariant subset $C$. If $A \preccurlyeq B$, then $A \preccurlyeq C$.
\end{lemma}

\begin{proof}
$A \preccurlyeq B$ means that $B \cap \mathcal{UE}(A)$ is non-empty. Moreover $B \cap \mathcal{UE}(A)$ is also closed and $G$-invariant thanks to Lemma \ref{lem-env-closed}. So by Zorn's lemma we may find a minimal closed $G$-invariant subset inside $B \cap \mathcal{UE}(A)$. By uniqueness the latter has to coincide with $C$, so that $C \subset \mathcal{UE}(A)$. In particular $A \preccurlyeq C$.
\end{proof}

\begin{prop} \label{prop-order-minimal}
If $\H,\K \in \urs(G)$, then $\H \preccurlyeq \K$ if and only if every element of $\H$ is contained in some element of $\K$ and every element of $\K$ contains some element of $\H$.
\end{prop}

\begin{proof}
Clearly it is enough to prove the direct implication. Since $\H \preccurlyeq \K$, $\H$ intersects the lower envelope of $\K$. Since $\mathcal{LE}(\K)$ is closed and $G$-invariant according to Lemma \ref{lem-env-closed}, one must have $\H \subset \mathcal{LE}(\K)$ by minimality of $\H$. A similar argument shows that $\
\K \subset \mathcal{UE}(\H)$, hence the conclusion. \qedhere \end{proof}

\begin{cor}\label{cor-order-URS}
The relation $\preccurlyeq$ is a partial order on the set $\urs(G)$.

\end{cor}

\begin{proof}
 Reflexivity is trivial. Transitivity follows from Proposition \ref{prop-order-minimal}. Let us prove that $\preccurlyeq$ is antisymmetric. Assume that $\H,\K \in \urs(G)$ are such that $\H \preccurlyeq \K$ and $\K \preccurlyeq \H$. By Zorn's lemma, we may find $H_0 \in \H$ that is maximal for the inclusion among elements of $\H$ (note that an increasing sequence of subgroups converges to its union in the Chabauty topology, thus an ascending union of elements of $\H$ belongs to $\H$ since $\H$ is closed). Applying Proposition \ref{prop-order-minimal} twice provides us with $K \in \K$ and $H_1 \in \H$ such that $H_0 \leq K\leq H_1$. By maximality of $H_0$ this implies $H_0=H_1$, and hence $H_0=K$. In particular $\H \cap \K$ is non-empty, and by minimality one must have $\H=\K$. 
\end{proof}

Recall that the amenable radical of $G$ is a normal amenable subgroup $R_a$ that contains any other normal amenable subgroup. By a result of Bader--Duchesne--L\'{e}cureux \cite{B-D-L}, every $\mu \in \mathrm{IRS}(G)$ supported on amenable subgroups satisfies $\mu(\sub(R_a)) = 1$.

The following result shows that every countable group $G$ admits an amenable URS that is larger than any other amenable URS. This URS is called \textbf{the Furstenberg URS} of the group $G$ (see Proposition \ref{prop-A_G-other-action} for the choice of this terminology). The Furstenberg URS can be seen as a \enquote{URS version} of the amenable radical for the study of the action $G \acts \sub(G)$, in the sense that it plays a role in the topological setting similar to the amenable radical in the measured setting.


When completing this work, we learn that  U.\ Bader and P-E.\ Caprace independently obtained a similar statement in the setting of locally compact groups.

\begin{thm} \label{thm-max-urs-moy}
Let $G$ be a countable group. Then there exists a unique amenable $\A_G\in \urs(G)$ such that $\H\preccurlyeq \A_G$ for every amenable $\H\in \urs(G)$.
\end{thm}

\begin{proof}
Let $\M(G)\subset \ell^\infty(G)^*$ be the set of all means on $G$. For the weak-* topology, $\M(G)$ is a convex compact $G$-space. The subgroups of $G$ which fix a point in $\M(G)$ are exactly the amenable subgroups of $G$. Namely any amenable subgroup of $G$ fixes a point in $\M(G)$ by the fixed point characterisation of amenability; conversely every $H\in \sub(G)$ that fixes a point in $\M(G)$ must be amenable, as its acts on $G$ with an invariant mean and trivial stabilizers.

By Zorn's lemma, we may choose a non-empty $G$-invariant convex closed subset $C\subset \M(G)$ which is minimal with respect to inclusion. Denote by $X\subset C$ the closure of the set of extreme points of $C$, which is a $G$-invariant subset. The choice of $C$ implies that $X$ is a minimal closed invariant subset, and is the unique minimal closed invariant subset of $C$, see \cite[Theorem III.2.3]{Glas-book}. Let $\A_G = \S_G(X)$ be the stabilizer URS associated to $G \acts X$. Recall that $\S_G(X)$ is the closure in $\sub(G)$ of $\{G_x \, : \, x\in X_0\}$, where $X_0$ is the domain of continuity of the stabilizer map. We claim that this URS verifies the desired conclusion. To see this, let $\H\in \urs(G)$ be another amenable URS, and let $H\in \H$. Since $H$ is amenable, $H$ fixes a point $c\in C$. Since $X$ is the only minimal closed invariant subset of $C$, the closure of the $G$-orbit of $c$ must contain $X$. In particular, there is a net $(g_i)_{i\in I}$ in $G$ such that $g_i(c)$ converges to a point $x\in X_0$. Then any cluster point of the net $(g_i H g_i^{-1})_{
i\in I}$ is contained in $G_x$. Since $G_x \in \A_G$, this shows that $\H\preccurlyeq \A_G$ as desired.

The uniqueness is clear: if $\mathcal{A}_1, \mathcal{A}_2 \in \urs(G)$ both satisfy the conclusion, then $\mathcal{A}_1 \preccurlyeq \mathcal{A}_2$ and $\mathcal{A}_2 \preccurlyeq \mathcal{A}_1$, so that $\mathcal{A}_1= \mathcal{A}_2$ by Corollary \ref{cor-order-URS}. \qedhere
\end{proof}

\begin{defin}
$\A_G$ will be called the \textbf{Furstenberg URS} of the group $G$.
\end{defin}

\begin{prop} \label{prop-properties-A_G}
Let $G$ be a countable group. Then the following hold:
\begin{enumerate}[label=(\roman*)]
\item $\mathcal{A}_{G}$ is invariant under the action of $\aut(G)$ on $\sub(G)$.
\item $\bigcap_{H \in \mathcal{A}_{G}} H = R_a$, where $R_a$ is the amenable radical of $G$.
\item $\mathcal{A}_G$ is either a singleton (in which case it is the amenable radical of $G$), or uncountable.
\item The kernel of the action $G \acts \mathcal{A}_{G}$ is a characteristic subgroup of $G$, which contains the amenable radical of $G$.
\end{enumerate}
\end{prop}

\begin{proof}
We first prove (i). The group $\aut(G)$ acts on the set $\urs_a(G)$ of amenable URS's of $G$ by preserving the order $\preccurlyeq$. Since $\mathcal{A}_G$ is the greatest element of $\urs_a(G)$, it follows that $\aut(G)$ must preserve $\mathcal{A}_G$.

For (ii), remark that $R = \cap_{H \in \mathcal{A}_{G}} H$ is a normal amenable subgroup of $G$, so that $R \leq R_a$. Conversely since the singleton $R_a$ is an amenable URS of $G$, we must have $R_a \preccurlyeq \mathcal{A}_{G}$. By Proposition \ref{prop-order-minimal} this means that every element of $\mathcal{A}_{G}$ contains $R_a$, so $R_a \leq R$. 

In order to prove (iii), we assume that $\mathcal{A}_G$ is countable, and we show that it must be a singleton. By the Baire category theorem, $\mathcal{A}_G$ must have an isolated point. Now by minimality all points must be isolated, and by compactness $\mathcal{A}_G$ has to be finite. This means that there is an amenable subgroup $H \in \sub(G)$ with finitely many conjugates. This property easily implies that the normal closure $N$ of $H$ is also amenable. Since $\mathcal{A}_G$ is larger than any other amenable URS, one must have $N \leq H$. Since the other inclusion is clear by definition, it follows that $H=N$, so $\mathcal{A}_G$ is a singleton. This proves (iii).

Finally (iv) follows immediately from (i) and (ii).
\end{proof}

\begin{remark}
There exist countable groups $G$ with an uncountable amenable URS, but whose Furstenberg URS is a singleton. For example $\A_G$ is the singleton $\{G\}$ whenever $G$ is amenable, but there are many examples of amenable groups admitting an uncountable URS.
\end{remark}

The following result establishes basic stability properties of the Furstenberg URS.

\begin{prop} \label{prop-stability-A_G}
The following hold:
\begin{enumerate}[label=(\roman*)]
\item If $G_1,G_2$ are countable groups, then the natural map $i : \sub(G_1) \times \sub(G_2) \rightarrow \sub(G_1 \times G_2)$ induces an equivariant isomorphism between $\mathcal{A}_{G_1} \times \mathcal{A}_{G_2}$ and $\mathcal{A}_{G_1 \times G_2}$.
\item If $N \unlhd G$ is a normal amenable subgroup of $G$, and if $\sub_N(G)$ is the set of subgroups of $G$ containing $N$, then the natural map $\varphi: \sub(G/N) \rightarrow \sub_N(G)$ induces an equivariant homeomorphism between $\mathcal{A}_{G/N}$ and $\mathcal{A}_{G}$.
\end{enumerate}
\end{prop}

\begin{proof}
(i). The map $i$ is continuous, and $i(H_1,H_2)$ is amenable as soon as $H_1$ and $H_2$ are amenable. Since moreover the action of $G_1 \times G_2$ on $\mathcal{A}_{G_1} \times \mathcal{A}_{G_2}$ is minimal, this shows that $i(\mathcal{A}_{G_1} \times \mathcal{A}_{G_2})$ is an amenable URS of $G_1 \times G_2$. Moreover one easily check that any amenable $\H \in \urs(G_1 \times G_2)$ satisfies $\H \preccurlyeq i(\mathcal{A}_{G_1} \times \mathcal{A}_{G_2})$, hence the conclusion. 

(ii). It is not hard to check that $\varphi$ is an equivariant homeomorphism. Since $N$ is amenable, $\varphi$ sends the set of amenable subgroups of $G/N$ onto the set of amenable subgroups of $G$ containing $N$. Therefore $\varphi(\mathcal{A}_{G/N})$ is an amenable URS of $G$, so that $\varphi(\mathcal{A}_{G/N}) \preccurlyeq \mathcal{A}_{G}$ by definition of $\mathcal{A}_{G}$. Now since $N$ is normal and amenable, every $H \in \mathcal{A}_{G}$ contains $N$ by Proposition \ref{prop-properties-A_G}, (ii). This means that $\mathcal{A}_{G} \subset \sub_N(G)$. Therefore $\varphi^{-1}(\mathcal{A}_{G})$ is an amenable URS of $G/N$, and it follows that $\varphi^{-1}(\mathcal{A}_{G}) \preccurlyeq \mathcal{A}_{G/N}$.
\end{proof}


Recall that Furstenberg showed that for every countable (and more generally locally compact) group $G$, there exists a compact space $\partial_F G$ with a boundary action $G \acts \partial_F G$ satisfying the following universal property: for every boundary action $G\acts X$, there exists a continuous surjective $G$-equivariant map $\partial_F G \to X$ \cite[Proposition 4.6]{Furst-bound-th}. The space $\partial_F G$ is called the \textbf{Furstenberg boundary} of $G$. It is unique up to $G$-equivariant homeomorphism.

The amenable radical of $G$ is exactly the kernel of the action of $G$ on its Furstenberg boundary \cite[Proposition 7]{Furman-boundary}. The second statement of the following proposition says that the Furstenberg URS of $G$ is the stabilizer URS associated to the action of $G$ on its Furstenberg boundary (which already played an important role in \cite{Kenn}). 


\begin{prop} \label{prop-A_G-other-action}
Let $G$ be a countable group.
\begin{enumerate}[label=(\roman*)]
\item If $G \acts X$ is a boundary action, then $\mathcal{A}_G \preccurlyeq \S_G(X)$. If moreover there exists $x\in X$ such that $G_x$ is amenable, then $\mathcal{A}_G = \S_G(X)$.
\item $\mathcal{A}_G$ is the collection of point stabilizers for the action of $G$ on its Furstenberg boundary.
\item The action $G \acts \mathcal{A}_G$ is a boundary action.
\end{enumerate}
\end{prop}

\begin{proof}
We first prove (i). Let $G \acts X$ be a boundary action, let $\mathcal{H}$ be an arbitrary amenable URS of $G$, and let $H \in \H$. Since $H$ is amenable, there exists a probability measure $\mu$ on $X$ that is fixed by $H$. By minimality and strong proximality, there exists a net $(g_i)$ such that $\mu_i = g_i \cdot \mu$ converges to some $\delta_x$. Write $H_i = g_i H g_i^{-1}$. Up to passing to a subnet, we may assume that $(H_{i})$ converges to some $K \in \H$. Since the net $(\mu_i)$ converges to $\delta_x$, every element of $K$ must fix the point $x$, i.e.\ $K \leq G_x$. This shows that $\H$ is $\preccurlyeq$ than the closure in $\sub(G)$ of the collection of subgroups $G_y$, $y \in X$. Now this set contains $\S_G(X)$ as unique URS by Proposition \ref{prop-min-action-urs}, so by applying Lemma \ref{lem-prec-unique-min} we deduce that $\H \preccurlyeq \S_G(X)$. In particular this shows $\A_G \preccurlyeq \S_G(X)$ because by definition $\A_G$ is amenable.

The existence of $x \in X$ such that $G_x$ is amenable implies that $\S_G(X)$ is amenable, since the set of amenable subgroups is closed in $\sub(G)$ for every countable group $G$. Therefore one must have $\A_G \preccurlyeq \S_G(X)$ by definition of $\A_G$, and the equality $\S_G(X) = \A_G$ follows from the previous paragraph.

We shall now prove (ii). Recall that point stabilizers for the action of $G$ on its Furstenberg boundary are amenable (to see this it is enough to exhibit one boundary action with amenable stabilizer, and this follows by applying \cite[Theorem III.2.3]{Glas-book} to a minimal closed convex $G$-invariant subset $C \subset \M(G)$). Hence statement (i) implies $\S_G(\partial_F G) = \A_G$. This means that $\A_G$ is the closure in $\sub(G)$ of the set of $G_{x_0}$, for $x_0 \in X_0 \subset \partial_F G$, where $X_0$ is the domain of continuity of $\mathrm{Stab}: \partial_F G \rightarrow \sub(G)$. Now according to \cite[Lemma 3.3]{BKKO}, $\fix(g)$ is open in $\partial_F G$ for every $g \in G$, so it follows from Lemma \ref{lem-cont-stab} that the stabilizer map is continuous on $\partial_F G$, i.e.\ $X_0 = \partial_F G$. This shows that the map $\partial_F G \rightarrow \A_G$ is onto. Part (iii) immediately follows, since any factor of a boundary action is a boundary action.
\end{proof}

\begin{remark} \label{rem-AG}
$ $
\begin{enumerate}
\item The fact that the action of $G$ on the set of point stabilizers $G_x$, $x \in \partial_F G$, is a boundary action was observed in \cite[Remark 4.2]{Kenn}. 
\item  The boundary $\A_G$ does not coincide in general with the Furstenberg boundary $\partial_F G$. An important difference is that $\mathcal{A}_G$ is metrizable (as a subspace of $\sub(G)$), while $\partial_F G$ is extremally disconnected, and therefore is never metrizable (unless it is trivial) \cite[Remark 3.16]{KK}.
\item We have seen in the proof of Proposition \ref{prop-A_G-other-action} that the stabilizer map $\mathrm{Stab}: \partial_F G \rightarrow \sub(G)$, associated to the boundary action $G \acts \partial_F G$, is always continuous. This need not be true in general for the boundary action $G \acts \mathcal{A}_G$, as explained in Remark \ref{rmq-normalizer-not-cont}.
\end{enumerate} 
\end{remark}

Furstenberg proved that for every countable group $G$, the action $G \acts \partial_F G$ may be extended to an action of $\mathrm{Aut}(G)$ on $\partial_F G$, which is faithful whenever $G \acts \partial_F G$ is \cite[Theorem II.4.3]{Glas-book}. The following proposition, due to Breuillard--Kalantar--Kennedy--Ozawa, says that $\mathrm{Aut}(G) \acts \partial_F G$ is also free whenever $G \acts \partial_F G$ is.


\begin{prop} \label{prop-urs-aut(G)}
Let $G$ be a countable group. If $G$ acts freely on its Furstenberg boundary $\partial_F G$, then $\mathrm{Aut}(G) \acts \partial_F G$ is also free.

Equivalently, if $\A_G$ is trivial then so is $\A_{\mathrm{Aut}(G)}$.
\end{prop}

\begin{proof}
Remark that $G$ must have trivial center since $G \acts \partial_F G$ is free. Then apply \cite[Lemma 5.3]{BKKO} to $G$, viewed as a normal subgroup of $\mathrm{Aut}(G)$.
\end{proof}

It is proved in Theorem 1.4 in \cite{BKKO} that if $N$ is a normal subgroup of $G$, then $G$ acts freely on its Furstenberg boundary if and only if both $N$ and $C_N$ act freely on their Furstenberg boundaries, where $C_N$ is the centralizer of $N$ in $G$. In terms of uniformly recurrent subgroups, this means that $\A_G$ is trivial if and only if both $\A_N$ and $\A_{C_N}$ are trivial. The following proposition refines the connections between $\A_G$, $\A_N$ and $\A_{C_N}$.

\begin{prop}
Let $G$ be a countable group, and $N$ a normal subgroup of $G$.
\begin{enumerate}[label=(\roman*)]
\item If $\mathcal{A}_G$ is a singleton (respectively, is trivial), then $\mathcal{A}_N$ is a singleton (respectively, is trivial). 
\item Suppose $\mathcal{A}_N$ is trivial. Then $\mathcal{A}_G = \mathcal{A}_{C_N}$, where $C_N$ is the centralizer of $N$ in $G$.
\item Suppose $\mathcal{A}_N = R_a(N)$ is a singleton, and let $M = \left\{g \in G \, : \, [g,n] \in R_a(N) \, \, \forall n \in N \right\}$ be the preimage in $G$ of the centralizer of $N/R_a(N)$ in $G/R_a(N)$. Then $\mathcal{A}_G = \mathcal{A}_M$.
\item If $\mathcal{A}_N$ and $\mathcal{A}_{G/N}$ are singletons (respectively, are trivial), then $\mathcal{A}_G$ is a singleton (respectively, is trivial). 
\end{enumerate}
\end{prop}

\begin{proof}
In order to prove (i), we first assume that $\A_G$ is trivial. If $\mathcal{A}_N$ is not trivial, then  Proposition \ref{prop-properties-A_G}(i) implies that $\mathcal{A}_N$ is a non-trivial amenable URS of $G$, a contradiction. The case where $\A_G$ is a singleton follows according to statement (ii) of Proposition \ref{prop-stability-A_G}.

We shall now prove (ii). The map $G \rightarrow \mathrm{Aut}(N)$ induces an action of $G$ on $\partial_F N$. Since $\mathcal{A}_{N}$ is trivial, it follows from Proposition \ref{prop-urs-aut(G)} that point stabilizers in $G$ are all equal to the kernel of $G \rightarrow \mathrm{Aut}(N)$, which is exactly $C_N$. By Proposition \ref{prop-A_G-other-action}, this shows $\A_G \preccurlyeq C_N$. Therefore $\A_G$ is a closed amenable $C_N$-invariant subset of $\sub(C_N)$, and it follows that $\A_G \preccurlyeq \A_{C_N}$. On the other hand $\A_{C_N}$ is itself an amenable URS of $G$ by Proposition \ref{prop-properties-A_G}(i), hence $\A_{C_N}\preccurlyeq \A_G$, and the equality follows.

In order to prove (iii), write $G' = G/R_a(N)$ and $N' = N/R_a(N)$. By our assumption and Proposition \ref{prop-stability-A_G}(ii), we have that $\mathcal{A}_{N'}$ is trivial. So by applying statement (ii) we obtain $\mathcal{A}_{G'} = \mathcal{A}_{C_{N'}}$, and the conclusion follows by applying again Proposition \ref{prop-stability-A_G}.

It is enough to prove (iv) in the case of singletons, and we may assume that $N$ has trivial amenable radical by Proposition \ref{prop-stability-A_G}. Therefore by point (ii), we deduce that $\mathcal{A}_G = \mathcal{A}_{C_N}$. Denote by $\pi$ the canonical projection from $G$ to $G/N$. The group $\pi(C_N)$ is normal in $G/N$, so it follows from statement (i) of the proposition that $\mathcal{A}_{\pi(C_N)}$ is a singleton. Now $\pi(C_N) = C_N / N \cap C_N$ and $N \cap C_N$ is abelian, so the conclusion follows from Proposition \ref{prop-stability-A_G}.
\end{proof}

\subsection{Amenable uniformly recurrent subgroups and $C^\ast$-simplicity}

Recall that a group $G$ is said to be $C^\ast$-simple if its reduced $C^\ast$-algebra $\Cred(G)$ is simple. This is an important property, which can be characterized in terms of weak containment of representations of the group \cite{dlH-survey}. By work of Kalantar and Kennedy \cite{KK}, $C^*$-simplicity also admits the following dynamical characterization.


\begin{thm}[Kalantar-Kennedy] \label{thm-KK-furst}
A countable group $G$ is $C^\ast$-simple if and only if $G$ acts freely on its Furstenberg boundary.
\end{thm}

Combining Theorem \ref{thm-KK-furst} with the description of $\A_G$ given in Proposition \ref{prop-A_G-other-action} yields the following result from \cite{Kenn}.

\begin{thm}[Kennedy] \label{thm-cstar-chab}
A countable group $G$ is $C^\ast$-simple if and only if it has no non-trivial amenable uniformly recurrent subgroup; equivalently if and only if the conjugacy class of every amenable subgroup of $G$ accumulates at the trivial subgroup in $\sub(G)$. 
\end{thm}
	
	

\section{Micro-supported actions and uniformly recurrent subgroups} \label{sec-main-thms}

In this section $X$ is a Hausdorff space, and $G\le \homeo(X)$ is a countable group of homeomorphisms of $X$. Although we make a priori no global additional assumption on $X$, most of the results of this section are relevant only when $X$ has no isolated points. Note also that if $G \acts X$ is micro-supported, then $X$ cannot have isolated points.


\subsection{Preliminaries}

In this section we collect some preliminary results that will be used in the sequel.

\begin{lemma} \label{L: disjoint open}
Let $X$ be a Hausdorff space without isolated points. Let $g_1, \ldots, g_r$ be non-trivial homeomorphisms of $X$. Then there exist non-empty open subsets $U_1, \ldots, U_r\subset X$ such that the sets $U_1, \ldots, U_r, g_1(U_1), \ldots g_r(U_r)$ are pairwise disjoint.

Moreover given any $z\in X$, we may find a neighbourhood $W$ of $z$ and open sets $U_i$ as above such that $W$ is disjoint from all the $U_i$ and all the $g_j^{-1}(U_i)$.
\end{lemma}

\begin{proof}
First observe that every non-trivial homeomorphism of $X$ moves infinitely many points. To see this, let $f$ be a non-trivial homeomorphism, and let $x\in X$ such that $f(x)\neq x$. Since $X$ is Hausdorff, there exists a neighbourhood $U$ of $x$ such that $U\cap f(U)=\varnothing$. Since moreover $X$ has no isolated points, the open set $U$ is infinite, and every point of $U$ is moved by $f$.
 
We now prove the first sentence of the statement. Let $x_1$ be a point moved by $g_1$ and set $y_1=g_1(x_1)$. Since $g_2$ moves infinitely many points, it moves at least one point $x_2$ which is different from $x_1, y_1, g_2^{-1}(x_1), g_2^{-1}(y_1)$. Set $y_2=g_2(x_2)$. Proceeding by induction by avoiding at each step only finitely many points, we can find $x_1, \ldots x_r\in X$ such that $g_i(x_i)\neq x_i$ and the points $x_1, \ldots x_r, y_1=g_1(x_1), \ldots, y_r=g_r(x_r)$ are all distinct. Since $X$ is Hausdorff, there exist neighbourhoods $V_i$ of $x_i$ and $W_i$ of $y_i$ such that $V_1, \ldots, V_r, W_1, \ldots W_r$ are pairwise disjoint. Set $U_i=V_i\cap g^{-1}_i(W_i)$. Then $U_1, \ldots, U_r$ verify the conclusion of the lemma. 

To prove the last sentence, observe that if $z\in X$ is given, we may choose carefully the points $x_i$ so that they are all different from $z$ and from $g_j(z)$ for all $j=1, \ldots, r$. Then a similar argument gives the conclusion.\qedhere
\end{proof}

We will make use of the following group theoretic lemma, which is exactly Lemma 4.1 in \cite{BHNeumann}. We include a short proof based on random walk.

\begin{lemma}[B.H. Neumann] \label{L: finite union}
Let $\Gamma$ be a countable group that can be written as a union $\cup_{i=1}^r \Delta_i\gamma_i$ of $r$ cosets of subgroups $\Delta_i \leq \Gamma$. Then at least one of the subgroups $\Delta_\ell$ has index at most $r$.
\end{lemma}

\begin{proof}
Let $\mu$ be a symmetric, non-degenerate probability measure on $\Gamma$ such that $\mu(1)>0$ and let $g_n=h_1\cdots h_n$ be the corresponding random walk.
We have 
\[1=\mu^{*n}(\Gamma)\leq \sum_{i=1}^r \mu^{*n}(\Delta_i\gamma_i),\]
from which we deduce that there exists $\ell$ such that $\mu^{*n}(\Delta_\ell\gamma_\ell)\geq\frac{1}{r}$ holds for infinitely many $n$. The latter is the probability that the Markov chain $(\Delta_\ell g_n)$ on the coset space $\Delta_\ell\backslash \Gamma$ is equal to $\Delta_\ell\gamma_\ell$ at time $n$. The assumptions on $\mu$ imply that $(\Delta_\ell g_n)$ is an irreducible, aperiodic, reversible Markov chain on  $\Delta_\ell \backslash \Gamma$. It is easy to check that the counting measure on $\Delta_\ell \backslash \Gamma$ is a stationary measure for this Markov chain (this is a general fact about Markov chain on coset spaces induced by a random walks on the group). By the ergodic theorem for aperiodic Markov chains, the probability that $ \Delta_\ell g_n$ is equal to any given coset tends to  $1/[\Gamma : \Delta_\ell]$ (where the latter is set to be 0 if $[\Gamma : \Delta_\ell]=\infty$). Since we assumed that $\mu^{*n}(\Delta_\ell \gamma_\ell)\geq 1/r$ for infinitely many $n$'s, we deduce that $[\Gamma : \Delta_\ell]\leq r$. \qedhere
\end{proof}


\begin{lemma} \label{lem-deltagamma-1}
Let $\Gamma$ be a group that can be written as a finite union $\Gamma = \cup_{i=1}^r Y_i$ of subsets. Then there exists $\ell\in\{1, \ldots, r\}$ such that the subgroup
\[\Delta_\ell=\langle\gamma\delta^{-1}\, \mid\, \gamma, \delta\in Y_\ell\rangle\]
has index at most $r$.
\end{lemma}

\begin{proof}
Select one $\gamma_i\in Y_i$ for every $i$. Then we have $Y_i\subset \Delta_i\gamma_i$. Thus $\Gamma =\cup_{i=1}^r\Delta_i\gamma_i$, and at least one of the $\Delta_i$ has index at most $r$ by the previous lemma.
\end{proof}
  
We will need the following classical lemma, a proof of which can be found in  \cite[Lemma 4.1]{Nek-fp}.

\begin{lemma} \label{double-comm}
Let $X$ be a Hausdorff space, and let $G$ be a group of homeomorphisms of $X$. If $N$ is a non-trivial normal subgroup of $G$, there exists a non-empty open $U \subset X$ such that $[G_U,G_U] \leq N$. 
\end{lemma}

\subsection{From rigid stabilizers to uniformly recurrent subgroups}

We still denote by $G$ a countable group of homeomorphisms of a Hausdorff space $X$. In this subsection we show that many properties of the rigid stabilizers $G_U$ are inherited by URS's of the group $G$. This is a consequence of the following result.

\begin{thm}\label{thm-dycothomy-mostgeneral}
Let $G$ be a countable group of homeomorphisms of a Hausdorff space $X$. Then for every $H \in \sub(G)$, one of the following possibilities holds:
\begin{enumerate}[label=(\roman*)]
\item the closure of the conjugacy class $\mathcal{C}(H)$ in $\sub(G)$ contains the trivial subgroup;
\item there exists a non-empty open $U\subset X$ such that $H$ admits a finite index subgroup of $G_U$ as a subquotient.
\end{enumerate}
\end{thm}

Recall that a group $Q$ is a \textbf{subquotient} of a group $H$ if there exists a subgroup $K \leq H$ such that $Q$ is a quotient of $K$.

\medskip

Before giving the proof of Theorem \ref{thm-dycothomy-mostgeneral}, let us single out the following consequence.


\begin{cor}\label{cor-mostgeneral-Q}
Retain the notation of Theorem \ref{thm-dycothomy-mostgeneral}. If all rigid stabilizers $G_U$, $U$ non-empty and open in $X$, are non-amenable (respectively are not elementary amenable, contain free subgroups, are not virtually solvable), then every non-trivial uniformly recurrent subgroup of $G$ has the same property.
\end{cor}


In particular when combining Theorem \ref{thm-cstar-chab} together with Corollary \ref{cor-mostgeneral-Q}, we obtain the following result. Note that the second assertion follows from the fact that having non-amenable rigid stabilizers is stable under taking overgroups in $\homeo(X)$.

\begin{thm}\label{thm-non-amenab-rigid-stab}
Let $X$ be a Hausdorff space, and let $G$ be a countable group of homeomorphisms of $X$. Assume that for every non-empty open $U \subset X$, the group $G_U$ is non-amenable. Then $G$ is $C^\ast$-simple.

More generally every countable subgroup of $\homeo(X)$ containing $G$ is $C^*$-simple.
\end{thm}

\medskip

The end of this paragraph is dedicated to the proof of Theorem \ref{thm-dycothomy-mostgeneral}. We will actually prove the following more technical statement, which implies Theorem \ref{thm-dycothomy-mostgeneral}, and which will be used later on.

\begin{prop} \label{prop-technique}
Let $G$ be a countable group of homeomorphisms of $X$. Fix $z \in X$. Then for every $H \in \sub(G)$, one of the following possibilities holds:
\begin{enumerate}[label=(\roman*)]
\item the closure of the conjugacy class $\mathcal{C}(H)$ in $\sub(G)$ contains the trivial subgroup;
\item there exists a neighbourhood $W$ of $z$ such that for every $K \in \mathcal{C}(H)$, there exist a non-empty open $U\subset X$, a finite index subgroup $\Gamma \leq G_U$ and a subgroup $A \leq K$ with the following properties:
\begin{itemize}
\item $A$ fixes $W$ pointwise;
\item $A$ leaves  $U$ invariant, and for every $\gamma \in \Gamma $, there is $a \in A$ such that $a$ coincides with $\gamma$ on $U$.
\end{itemize}
\end{enumerate} 
\end{prop}

Before going into the proof, we shall explain why Proposition \ref{prop-technique} implies Theorem \ref{thm-dycothomy-mostgeneral}. Indeed, if $H \in \sub(G)$ does not contain $\left\{1\right\}$ in the closure of its conjugacy class, and if $U$, $\Gamma$ and $A$ are as in condition (ii) of Proposition \ref{prop-technique}, then the restriction to $U$ provides a quotient of $A$ that contains an isomorphic copy of $\Gamma$. Since the latter has finite index in $G_U$ and since $A$ is a subgroup of $H$, this proves condition (ii) of Theorem \ref{thm-dycothomy-mostgeneral}. Note that this argument is completely independent of the choice of the point $z$ and the neighbourhood $W$. However these will be important for the application of Proposition \ref{prop-technique} in the next subsection.

\begin{proof}[Proof of Proposition \ref{prop-technique}]
We assume that (i) does not hold, and we prove (ii). We may plainly assume that $X$ has no isolated points, since otherwise (ii) is obviously true.

According to Lemma \ref{L: accumulation}, there exists a finite subset $P=\{g_1, \ldots, g_r\}\subset G\setminus \{1\}$ such that all the conjugates of $P$ intersect $K$ for every $K \in \conj(H)$. Let $U_1,\ldots, U_r\subset X$ be some open sets as in the conclusion of Lemma \ref{L: disjoint open} applied to $P$, and let also $W$ be a neighbourhood of $z$ as in Lemma \ref{L: disjoint open}. For every $i=1,\ldots , r$ we will write $G_i = G_{U_i}$, the rigid stabilizer of $U_i$. 

Let $L$ be the subgroup of $G$ generated by all the $G_i$ for $i=1,\ldots r$. By construction the $G_i$ have disjoint supports, and hence pairwise commute. In particular we have a natural identification $L =G_1\times\cdots \times G_r$. Let us denote $\pi_i: L \to G_i$ the projection of $L$ to $G_i$.

Fix $K \in \mathcal{C}(H)$. By definition of $P$, every element of $G$ must conjugate at least one element of $P$ inside $K$. In particular we can write $L$ as a finite union of subsets $L=Y_{1}\cup \cdots \cup Y_{r}$, where $Y_i$ is the set of elements of $L$ that conjugate $g_i \in P$ inside $K$: \[ Y_{i} = \left\{\gamma \in L \, \mid \, \gamma g_i \gamma^{-1} \in K\right\}.\] 

 By Lemma \ref{lem-deltagamma-1}, we may find an index $\ell$ such that the subgroup
\[\Delta_\ell=\langle\gamma\delta^{-1}\, \mid \, \gamma, \delta\in Y_\ell\rangle\] 
 has finite index in $L$. We fix such a $\ell$, and for every $\gamma, \delta\in Y_{\ell}$, we consider the element 
 \[a_{\gamma, \delta} = (\gamma g_\ell^{-1} \gamma^{-1}) (\delta g_\ell\delta^{-1}).\] Being the product of two elements of $K$, $a_{\gamma, \delta}$ is an element of $K$. We denote by $A \le K$ the subgroup generated by all the elements $a_{\gamma, \delta}$ when $\gamma, \delta$ range over $Y_{\ell}$. 

We will now prove that (ii) holds with $U = U_\ell$ and $\Gamma = \pi_\ell(\Delta_\ell)\le G_U$. Note that $\Gamma$ is indeed of finite index in $G_U$ because $\Delta_\ell$ has finite index in $L$. Moreover observe that $\Gamma$ is generated by the $\pi_\ell(\gamma\delta^{-1})$ when $\gamma, \delta$ run in $Y_\ell$. By definition of the projection, this is just the element that coincides with $\gamma\delta^{-1}$ on $U$ and with the identity elsewhere. Hence proving the following lemma is enough to conclude.

\begin{lemma}\label{lem-intermediaire}
For every $\gamma, \delta \in Y_{\ell}$, the element $a_{\gamma, \delta}$ fixes pointwise $W$, leaves $U_{\ell}$ invariant and coincides with $\gamma\delta^{-1}$ in restriction to $U_{\ell}$.
\end{lemma}

\begin{proof}[Proof of Lemma \ref{lem-intermediaire}]
To prove the statement, let us first rewrite $a_{\gamma, \delta} = \gamma (g_\ell^{-1} \gamma^{-1} \delta g_\ell) \delta^{-1}$. Since the elements $\gamma,\delta$ belong to $L$, they leave every $U_i$ invariant and have support contained in the union $\cup_{i=1}^r U_i$. In particular $\gamma$ and $\delta$ leave $U_\ell$ invariant and act trivially on $W$ thanks to Lemma \ref{L: disjoint open}. Now the support of  $g_\ell^{-1} \gamma^{-1} \delta g_\ell$ is contained in $g_\ell^{-1} \left(\cup_i U_i \right)$, and again by Lemma \ref{L: disjoint open} the latter is disjoint from both $W$ and $U_\ell$. Therefore $a_{\gamma, \delta}$ is trivial on $W$ and acts as $\gamma\delta^{-1}$ on $U_\ell$, and the lemma is proved. \qedhere
\end{proof}
This concludes the proof of the proposition. \qedhere
\end{proof}


\subsection{Case of an extremely proximal action}
\label{SS: EP}
Recall that we say that a subset $Y \subset X$ is \textbf{compressible} if there exists a point $x \in X$ such that for every open $U \subset X$ containing $x$, there exists $g \in G$ such that $g(Y) \subset U$. The action of $G$ on $X$ is said to be \textbf{extremely proximal} if every closed $C \neq X$ is compressible.


\begin{thm} \label{thm-EP} Let $G$ be a countable group of homeomorphisms of $X$.
Assume that the action of $G$ on $X$ is extremely proximal. Let $H \in \sub(G)$ such that the closure of the conjugacy class $\mathcal{C}(H)$ in $\sub(G)$ does not contain the trivial subgroup. Then there exist a non-empty open $U\subset X$ and a finite index subgroup $\Gamma \leq G_U$ such that $H$ contains $[\Gamma,\Gamma]$.
\end{thm}


\begin{proof}
First note that if there is a non-empty open subset $U$ such that $G_U$ is finite, then the conclusion is trivially satisfied. Therefore in the proof we may assume that all the subgroups $G_U$ are infinite. 

We start by applying Proposition \ref{prop-technique} (with an arbitrary choice of $z$). Since $\mathcal{C}(H)$ does not accumulate at the trivial subgroup in $\sub(G)$ by assumption, we obtain the existence of a non-empty open $U_0 \subset X$, a finite index subgroup $\Gamma_0 \leq G_{U_0}$ and a subgroup $A_0 \leq H$ with the following properties: $A_0$ preserves $U_0$, and for every $\gamma \in \Gamma_0$ there exists $a \in A_0$ which coincides with $\gamma$ on $U_0$.

By extreme proximality of the action of $G$ on $X$, the closed subset $C = X \setminus U_0$ is compressible. We let $z_0 \in X$ be a point of $X$ all of whose neighbourhoods contain an element of the $G$-orbit of $C$. Now we apply Proposition \ref{prop-technique} a second time, this time by choosing $z=z_0$. We denote by $W$ a neighbourhood of $z_0$ as in condition (ii) in Proposition \ref{prop-technique}.

Fix some $g \in G$ such that $g(C) \subset W$. According to the conclusion of Proposition \ref{prop-technique} applied with $K = gHg^{-1}$, there must exist a non-empty open $U_1 \subset X$, a finite index subgroup $\Gamma_1 \leq G_{U_1}$ and a subgroup $B \leq gHg^{-1}$ such that $B$ acts trivially on $W$, and the action of any element of $\Gamma_1$ on $U_1$ can be realized by an element of $B$. Note that since $G_{U_1}$ is assumed to be infinite, the same holds for its finite index subgroup $\Gamma_1$, and a fortiori the subgroup $B$ is also infinite. 

Set $A_1 = g^{-1} B g\le H$. Since $B$ acts trivially on $W$, the subgroup $A_1$ acts trivially on $g^{-1}(W)$. Now $g^{-1}(W)$ contains $C = X \setminus U_0$, so it follows that $A_1$ is a subgroup of $G_{U_0}$. Note also that $A_1$ is a subgroup of $H$ by construction, so that we actually have $A_1 \leq H_{U_0}$. Now consider $A_2 = A_1 \cap \Gamma_0$. Since $\Gamma_0$ has finite index in $G_{U_0}$, the subgroup $A_2$ is of finite index in $A_1$. Now as noticed above, $B$ (and hence $A_1$) must be infinite, so we deduce that $A_2$ is also infinite.

We claim that $H_{U_0}$ contains $\langle\! \!\langle A_2^{\Gamma_0}\rangle \! \!\rangle$, the normal closure of $A_2$ in $\Gamma_0$. To show this, let $x \in A_2$ and $\gamma \in \Gamma_0$. According to the conclusion of Proposition \ref{prop-technique}, there is $a \in A_0$ which coincides with $\gamma$ on $U_0$. This implies that the element $axa^{-1}$ coincides with $\gamma x \gamma^{-1}$ on $U_0$. Moreover $axa^{-1}$ is trivial outside $U_0$ because $x \in A_2$ and $A_2$ acts trivially outside $U_0$. Since $\gamma x \gamma^{-1}$ is also trivial outside $U_0$ (since both $\gamma$ and $x$ are trivial outside $U_0$), it follows that the elements $axa^{-1}$ and $\gamma x \gamma^{-1}$ are actually equal. Since $A_0$ and $A_2$ are subgroups of $H$, the element $axa^{-1}$ belongs to $H$, and the claim is proved. 

In particular we have proved that $H_{U_0}$ contains a non-trivial normal subgroup of $\Gamma_0$ (since $A_2$ is non-trivial). According to Lemma \ref{double-comm} applied to $\Gamma_0$ acting on $U_0$, there exists a non-empty open $U \subset U_0$ such that the derived subgroup of $(\Gamma_0)_{U}$ is contained in $H_{U_0}$. Since moreover $\Gamma = (\Gamma_0)_{U}$ has finite index in $G_U$ because $\Gamma_0$ has finite index in $G_{U_0}$, we have proved that $\Gamma$ satisfies the desired conclusion. 
\end{proof}

\begin{cor} \label{cor-EP-acc-Gx}
Assume that the following conditions are satisfied: 
\begin{enumerate}[label=(\roman*)]
	\item $X$ is compact, and the action of $G$ on $X$ is extremely proximal;
	\item for every open $U \subset X$ and every finite index subgroup $\Gamma \leq G_U$, there exists an open $V \subset U$ such that $G_V \leq [\Gamma,\Gamma]$.
\end{enumerate}

If $H \in \sub(G)$ is such that $\mathcal{C}(H)$ does not accumulate at the trivial subgroup in $\sub(G)$, then there exists a point $x \in X$ such that $\mathcal{C}(H)$ accumulates at some overgroup of $G_x^0$.  
\end{cor}

\begin{proof}
Theorem \ref{thm-EP} shows that there is $U \subset X$ and a finite index subgroup $\Gamma \leq G_U$ such that $[\Gamma,\Gamma]$ is contained in $H$. Combining with assumption (ii), we obtain a non-empty open $V \subset X$ such that $K = G_V$ is contained in $H$. 

Since the action of $G$ on $X$ is extremely proximal, there exists $x \in X$ such that any neighbourhood of $x$ contains a $G$-translate of the complement of $V$. Let $\mathcal{W}$ be the collection of all neighbourhoods of $x$ in $X$. For every $W\in \mathcal{W}$, we let $g_W \in G$ such that $g_W(X \setminus V) \subset W$, and we write $K_W = g_W K g_W^{-1}$. Then $(K_W)_{W\in \mathcal{W}}$ is a net taking values in $\sub(G)$, where the index set $\mathcal{W}$ is partially ordered by reversed inclusion. We claim that every cluster point of this net must be an overgroup of $G_x^0$.  

Let $g \in G_x^0$. Since $g$ fixes pointwise some neighbourhood $W_0$ of $x$, it  acts trivially on $W$ for all $W\subset W_0$. In particular $g$ is trivial on  $g_W(X \setminus V)$, which exactly means that $g \in K_W$ for all $W \subset W_0$. This shows that any cluster point of the net $(K_W)$ must contain $g$, hence must contain $G_x^0$ since $g\in G_x^0$ was arbitrary. Since $K$ is contained in $H$, the conclusion then follows by considering any cluster point of the net $(g_W H g_W^{-1})\subset \Ccal(H)$.
\end{proof}

Recall (Proposition \ref{prop-min-action-urs}) that any minimal action of a countable group $G$ on a compact space $X$ gives rise to a URS of $G$, denoted $\mathcal{S}_G(X) \in \urs(G)$. More precisely, $\mathcal{S}_G(X)$ is the closure in $\sub(G)$ of the set of $G_{x_0}$, for $x_0 \in X_0$, where $X_0$ is the domain of continuity of the stabilizer map $\mathrm{Stab}: X \rightarrow \sub(G)$. 

The following result says that when the action is moreover assumed to be extremely proximal, under suitable assumptions, $\mathcal{S}_G(X)$ turns out to be smaller (with respect to the relation $\preccurlyeq$ defined in \textsection\ref{subsec-largest-urs}) than any other non-trivial URS. Recall that a minimal and extremely proximal action is called an extreme boundary action.

\begin{cor} \label{cor-unique-urs}
Assume that $X$ is compact, and that the following conditions are satisfied: 
\begin{enumerate}[label=(\roman*)]
	\item $G \acts X$ is an extreme boundary action;
	\item for every open $U \subset X$ and every finite index subgroup $\Gamma \leq G_U$, there exists an open $V \subset U$ such that $G_V \leq [\Gamma,\Gamma]$.
\end{enumerate}
Then $\mathcal{S}_G(X) \preccurlyeq \H$ for every non-trivial $\H \in \urs(G)$. 

In particular if there is $H \in \mathcal{S}_G(X)$ which is a maximal subgroup of $G$, then the only uniformly recurrent subgroups of $G$ are $1$, $\mathcal{S}_G(X)$, and $G$.
\end{cor}

\begin{proof}
Let $\H$ be a non-trivial URS of $G$. According to Corollary \ref{cor-EP-acc-Gx}, there exists $K \in \H$ and $x \in X$ such that $G_x^0 \leq K$. Let $x_0 \in X_0$. By minimality of the action of $G$ on $X$, there exists a net $(g_i)_{i\in I}$ such that $g_i(x)$ converges to $x_0$. Since $G_{x_0}^0 = G_{x_0}$ because $x_0 \in X_0$, it is an easy verification to show that the net $(g_i G_x^0 g_i^{-1})_{i\in I}$ converges to $G_{x_0}$. This shows that there is $L \in \H$ such that  $G_{x_0} \le L$, and the first statement is proved. The second statement immediately follows.
\end{proof}

\subsection{Extreme boundaries and amenable URS's}
 
Recall from Theorem \ref{thm-max-urs-moy} that every countable group $G$ admits a largest amenable URS, denoted $\A_G$, and that $\A_G$ coincides with the set of point stabilizers for the action of $G$ on its Furstenberg boundary (Proposition \ref{prop-A_G-other-action}). The following result shows that $\A_G$ can be explicitly identified as soon as one is given a faithful extreme boundary action of $G$.
 

\begin{thm}\label{thm-AG-extreme}
Let $G$ be a countable group, and $G\acts X$ be a faithful extreme boundary action. Then:
\begin{enumerate}[label=(\roman*)]
\item if $G_x$ is amenable for some $x\in X$, then $\A_G=\S_G(X)$;
\item if $G_x$ is non-amenable for all $x\in X$, then $\A_G$ is trivial.
\end{enumerate}
\end{thm}


\begin{proof}
Part (i) was already proven in Proposition~\ref{prop-A_G-other-action}, so we only have to prove (ii). Choose $x\in X$ such that $G_x=G^0_x$. Since amenability is preserved by direct limits, $G_x$ must have a non-amenable finitely generated subgroup, and it follows that there exists a closed subset $C \subset X$ not containing $x$ such that $G_{C}$ is non-amenable. Let $U\subset X$ be an arbitrary non-empty open set. By minimality and extreme proximality, there exists $g\in G$ such that $g(C)\subset U$. This implies that $gG_{C}g^{-1}$ is a subgroup of $G_U$, which is therefore non-amenable. Since $U$ was arbitrary, by Corollary~\ref{cor-mostgeneral-Q} it follows that $G$ has no non-trivial amenable URS's. \qedhere
\end{proof}


In particular the stabilizers of any faithful extreme boundary action characterize the $C^\ast$-simplicity of $G$:

\begin{cor}
Let $G$ be a countable group, and $G\acts X$ be a faithful extreme boundary action. Then $G$ is $C\ast$-simple if and only if one of the following possibilities holds:
\begin{enumerate}[label=(\roman*)]
\item the action is topologically free;
\item the point stabilizers are non-amenable.
\end{enumerate}
\end{cor}

\begin{proof}
By Theorem \ref{thm-cstar-chab} $G$ is $C^\ast$-simple if and only if $\A_G$ is trivial. Hence the statement follows from Theorem \ref{thm-AG-extreme}.\qedhere
\end{proof}

\begin{example} \label{exe-stab-nonamenable}
If $G\acts X$ is only required to be a boundary action, the non-amenability of point stabilizers no longer implies the $C^\ast$-simplicity of $G$. Indeed let $G_i\acts X_i$, $i=1,2$, be two faithful boundary actions which are not topologically free, and such that the point stabilizers of $G_1 \acts X_1$ are amenable (in particular, $G_1$ is not $C^\ast$-simple) but those of $G_2\acts X_2$ are not amenable. One can take for instance $G_1$ to be a group $G(F, F')$ (see \textsection\ref{subsec-G(F,F')}) acting on the boundary of the tree, and $G_2$ to be Thompson's group $V$ acting on the Cantor set. Consider $G=G_1\times G_2$, acting on $X=X_1\times X_2$ in the natural way. Then $G\acts X$ is a faithful boundary action, whose stabilizers are non-amenable since the stabilizers of $G_2\acts X_2$ already have this property. However $G$ is not $C^\ast$-simple since $G_1$ is not $C^\ast$-simple and is normal in $G$ \cite{BKKO}. Note that the action $G\acts X$ is not extremely proximal, since closed subsets of the form $X_1\times\{x\}$ are not compressible.
\end{example}

\section{Applications} \label{sec-applications}

\subsection{Thompson's groups}

Recall that Thompson's group $T$ is the group of orientation preserving homeomorphisms of $\Sbb^1=\R/\Z$ which are piecewise linear, with only finitely many breakpoints, all at dyadic rationals, and slopes in $2^{\mathbb{Z}}$. Thompson's group $F$ is the subgroup of $T$ that stabilizes the point $0 \in \Sbb^1$. 

To define Thompson's group $V$, we need the following notation: the binary Cantor set is the space $\Ccal=\{0, 1\}^\N$ of (infinite) binary sequences, endowed with the product topology. For every finite word $w$ in the alphabet $\{0, 1\}$ we denote $C_w\subset \Ccal$ the cylinder subset defined by $w$, which consists of those sequences that have $w$ as an initial prefix. Thompson's group $V$ is the group of homeomorphisms of $\Ccal$ consisting of elements $g$ for which there exist two cylinder partitions $\left\{C_{w_1},\ldots, C_{w_n}\right\}$ and $\left\{C_{z_1},\ldots, C_{z_n}\right\}$ of $\Ccal$ such that $g(w_ix)=z_i x$ for every $i$ and every binary sequence $x$.






\subsubsection{$C^*$-simplicity for Thompson's groups}

In this paragraph we completely elucidate the connections between the problems of the amenability of $F$ and the $C^*$-simplicity of $F$ and $T$ (Corollary \ref{cor-F-F-T}), and we prove the $C^*$-simplicity of the group $V$ and some of its relatives (see Theorem \ref{thm-V-and-relatives}).





\medskip

\begin{thm} \label{thm-F-overgroups-circle}
Suppose that $F$ is non-amenable. Then any countable subgroup $G \leq \homeo(\Sbb^1)$ containing $F$ must be $C^\ast$-simple. 
\end{thm}

\begin{proof}
Any rigid stabilizer $F_U$, with $U$ a non-empty open subset of $\Sbb^1$, contains an isomorphic copy of $F$ by Lemma 4.4 from \cite{C-F-P}, and is therefore non-amenable by our assumption. Therefore me may apply Theorem \ref{thm-non-amenab-rigid-stab}, from which the conclusion follows.
\end{proof}

Recall that Haagerup and Olesen proved in \cite{H-O} that if the group $T$ is $C^*$-simple, then the group $F$ has to be non-amenable. This result also appeared in \cite{BKKO}, where a partial converse is obtained, namely that if $T$ is not $C^*$-simple, then $F$ is not $C^*$-simple either \cite{BKKO}. The question whether the exact converse of Haagerup-Olesen's result holds was considered in \cite{B-J, Bleak-normalish}. The following result answers the question positively, and also says that the non-amenability of $F$ is also equivalent to its $C^*$-simplicity.


\begin{cor} \label{cor-F-F-T}
The following statements are equivalent:
\begin{enumerate}[label=(\roman*)]
\item The group $F$ is non-amenable;
\item The group $F$ is $C^*$-simple;
\item The group $T$  is $C^*$-simple.
\end{enumerate}
\end{cor} 

\begin{proof}
That (iii) implies (i) was proved in \cite{H-O}, and the implication from (ii) to (i) follows from a general argument \cite[Proposition 3]{dlH-survey}. That (i) implies both (ii) and (iii) is consequence of Theorem \ref{thm-F-overgroups-circle}.   
\end{proof}

\medskip

We now observe that the results obtained in Section \ref{sec-main-thms} may also be applied to other interesting groups related to Thompson's group $F$. We note that there is now a multitude of \enquote{Thompson-like} groups in the literature, for which similar arguments could be applied. We certainly do not try to be exhaustive here, and only give a few examples which further illustrate the results from Section \ref{sec-main-thms}.

Consider the group $\mathrm{PL}_2(\mathbb{R})$ of homeomorphisms of the real line which are piecewise linear with a discrete set of breakpoints (all of them dyadic rationals), with slopes in $2^{\mathbb{Z}}$ and which preserve the set of dyadic rationals. Thompson's group $F$ is well known to have a faithful representation $\rho: F \rightarrow \mathrm{PL}_2(\mathbb{R})$ (which is topologically conjugate to its standard action on the open interval $(0, 1)$), whose image $\rho(F)$ consists of those elements $g \in \mathrm{PL}_2(\mathbb{R})$ for which there exist $A >0$ and $m,n \in \mathbb{Z}$ such that $g(x) = x+m$ for every $x \leq A$ and $g(x) = x+n$ for every $x \geq A$. For the sake of simplicity we will still denote by $F$ the image of $\rho$, and when talking about $F$ inside $\mathrm{PL}_2(\mathbb{R})$ we will always implicitly refer to this representation. 

The following result is another formulation of Theorem \ref{thm-F-overgroups-circle}.

\begin{thm} \label{thm-F-overgroups}
Suppose that $F$ is non-amenable. Then any countable subgroup $G \leq \homeo(\mathbb{R})$ that contains $F$ must be $C^\ast$-simple. 
\end{thm}


We consider two countable groups of homeomorphisms of the real line that contain $F$. The first is the normalizer of $F$ in $\mathrm{PL}_2(\mathbb{R})$, which, by work of Brin \cite{Brin-cham}, turns out to be isomorphic to the group $\mathrm{Aut}(F)$ of automorphisms of $F$. A second example is the group $\mathrm{Comm}(F)$ of abstract commensurators of the group $F$, which was explicitly described inside $\mathrm{PL}_2(\mathbb{R})$ by Burillo, Cleary and R\"{o}ver \cite{Bur-Clea-Rov-1}. 


\begin{cor} The following statements are equivalent:
\begin{enumerate}[label=(\roman*)]
\item The group $F$ is non-amenable;
\item The automorphism group $\mathrm{Aut}(F)$ is $C^*$-simple;
\item The abstract commensurator group $\mathrm{Comm}(F)$ is $C^*$-simple.
\end{enumerate}
\end{cor}

\begin{proof}
The fact that (i) implies (ii) and (iii) is a consequence of Theorem \ref{thm-F-overgroups} thanks to the identifications of $\mathrm{Aut}(F)$ and $\mathrm{Comm}(F)$ (given respectively in \cite{Brin-cham} and \cite{Bur-Clea-Rov-1}) with overgroups of $F$ inside $\mathrm{PL}_2(\mathbb{R})$. Being centerless, $F$ embeds as a normal subgroup in $\mathrm{Aut}(F)$, so it clear that $F$ cannot be amenable if $\mathrm{Aut}(F)$ is $C^*$-simple. So (ii) implies (i). Finally it remains to see that (iii) also implies (i). According to \cite[Theorem 1]{Bur-Clea-Rov-2}, the group $[F,F]$ appears as a subnormal subgroup of $\mathrm{Comm}(F)$. Since being $C^*$-simple is inherited by normal subgroups \cite[Theorem 1.4]{BKKO}, it follows that if $\mathrm{Comm}(F)$ was $C^*$-simple then the same would be true for $[F,F]$. In particular $[F,F]$ would not be amenable, so $F$ would not be amenable either.
\end{proof}

\medskip

While the question of $C^*$-simplicity of the groups $F$ and $T$ remains open, the arguments developed in this paper allow to obtain the $C^*$-simplicity of the group $V$. Indeed, since the rigid stabilizer of a cylinder for the action of $V$ on the binary Cantor set is easily seen to be isomorphic to $V$, it follows that any rigid stabilizer $V_U$, for $U$ a non-empty open subset of $\Ccal$, is non-amenable (because the group $V$ contains non-abelian free subgroups). Therefore Theorem \ref{thm-non-amenab-rigid-stab} applies and shows that $V$ is $C^*$-simple. The exact same argument applies to the higher-dimensional groups $nV$, constructed by Brin in \cite{Brin-nV}, for the action of $nV$ on the Cantor $n$-cube. 

Theorem \ref{thm-non-amenab-rigid-stab} further implies that any countable subgroup of $\homeo(\Ccal)$ containing $V$ is $C^*$-simple. Interesting examples of such groups are the groups $V_G$, associated to a self-similar group $G$, considered by Nekrashevych in \cite{Nek-cuntz, Nek-fp}. 

The following result summarizes the above discussion.

\begin{thm} \label{thm-V-and-relatives}
The following groups are $C^*$-simple:
\begin{enumerate}[label=(\roman*)]
\item Thompson's group $V$;
\item the higher-dimensional groups $nV$, $n \geq 2$;
\item all the groups $V_G$, where $G$ is a countable self-similar group.
\end{enumerate}
\end{thm}

\subsubsection{Classification of uniformly recurrent subgroups}

In this paragraph we completely classify the URS's of the groups $F$, $T$ and $V$.

\begin{thm}[Classification of the URS's of Thompson's groups] \label{T: URS Thompson}
$ $
\begin{enumerate}[label=(\roman*)]
\item The only URS's of Thompson's group $F$ are the normal subgroups. The derived subgroup $[F,F]$ has no URS other than $1$ and $[F,F]$.
\item The URS's of Thompson's group $T$ are $1, T$ and the stabilizer URS arising from its action on the circle.
\item The URS's of Thompson's group $V$ are $1, V$ and the stabilizer URS arising from its action on the binary Cantor set.
\end{enumerate}
\end{thm}

\begin{proof}
Let us first prove (i). We consider the action of the group $F$ on the circle $\mathbb{S}^1 = [0,1] / \! \! \! \sim$, where $\sim$ identifies the points $0$ and $1$. For simplicity we will still denote by $0$ the image of $0$ in $\mathbb{S}^1$. Recall that the derived subgroup of $F$ consists exactly of those elements of $F$ that act trivially on a neighbourhood of $0$ in $\mathbb{S}^1$ \cite{C-F-P}. We need some preliminary lemmas.
\medskip

\begin{lemma} \label{lem-F-ep}
The action of $[F,F]$ on $\mathbb{S}^1$ is extremely proximal. 
\end{lemma}

\begin{proof}
We show that every proper closed $C \subset \mathbb{S}^1$ is compressible. Without loss of generality, we may assume that $C$ is contained in the complement of an open interval $]a,b[$, where $a,b$ are dyadic numbers and $0 < a < b < 1$. We let $]\beta, \alpha[$ be an open interval of $\mathbb{S}^1$ containing the point $0$, and we show that there is a $[F,F]$-translate of $C$ inside $]\beta, \alpha[$. Upon reducing $]\beta,\alpha[$ if necessary, we may assume that $\alpha, \beta$ are dyadic numbers and that $\alpha' = \alpha /2 < a$ and $\beta' = (1+ \beta) /2 > b$. We easily see that for $n$ large enough, any homeomorphism of $\mathbb{S}^1$ acting trivially on $[0,\alpha']$ and $[\beta',1]$, acting like $x \mapsto 2^{-n}x + \alpha' (1-2^{-n})$ on $[\alpha',a]$ and like $x \mapsto 2^{-n}x + \beta' (1-2^{-n})$ on $[b,\beta']$; will send $C$ into $]\alpha,\beta[$. Since such homeomorphisms can be found in the group $[F,F]$, the statement follows.
\end{proof}

\begin{lemma} \label{lem-rigid-stab-F}
For every dyadic numbers $0 < \alpha < \beta < 1$ with $\beta - \alpha \in 2^\mathbb{Z}$ and every finite index subgroup $\Gamma \leq F_{[\alpha, \beta]}$, there exist dyadic numbers $0 < a < b < 1$ such that $F_{[a,b]}$ is contained in the derived subgroup of $\Gamma$.
\end{lemma}

\begin{proof}
The subgroup $F_{[\alpha, \beta]}$ is easily seen to be isomorphic to $F$ \cite[Lemma 4.4]{C-F-P}. In particular $F_{[\alpha, \beta]}$ has a simple derived subgroup $N$, and $N$ is contained in any finite index subgroup. Therefore the derived subgroup of $\Gamma$ must contain $N$, and since $N$ is exactly the set of elements which are trivial on neighbourhoods of $\alpha$ and $\beta$, the conclusion actually holds for every dyadic numbers $a,b$ such that $\alpha < a < b < \beta$.
\end{proof}

\begin{lemma} \label{prop-lem-acc-F'}
Let $a,b$ be dyadic numbers such that $0 < a < b < 1$, and let $H$ be the rigid stabilizer of $[a,b]$ in $F$ (which is a subgroup of $[F,F]$). Then $[F,F]$ is an accumulation point of $\conj(H)$ in $\sub([F,F])$.
\end{lemma}

\begin{proof}
We choose $(g_n) \in [F,F]$ such that the sequence $(g_n([a,b]))$ is increasing and ascends to $\mathbb{S}^1 \setminus \left\{0\right\}$, and we denote $H_n = g_n H g_n^{-1}$. It is immediate to check that $(H_n)$ converges to the subgroup of $F$ consisting of elements which are trivial in a neighbourhood of $0$ in $\mathbb{S}^1$. Since the latter subgroup is exactly $[F,F]$, the proof is complete.
\end{proof}

We are now ready to prove that the only URS's of $[F,F]$ are $1$ and $[F,F]$. Write $G=[F,F]$, and let $\H \in \urs(G)$ be a non-trivial URS and $H \in \H$. Since the action of $G$ on $\mathbb{S}^1$ is extremely proximal by Lemma \ref{lem-F-ep}, Theorem \ref{thm-EP} implies that we may find dyadic numbers $0 < \alpha < \beta < 1$ such that $H$ contains the derived subgroup of a finite index subgroup of the rigid stabilizer $G_{[\alpha, \beta]}$. Now the group $G_{[\alpha, \beta]}$ is equal $F_{[\alpha, \beta]}$, so by Lemma \ref{lem-rigid-stab-F} we may find dyadic numbers $0 < a < b < 1$ such that $H$ contains the rigid stabilizer $F_{[a,b]}$. According to Lemma \ref{prop-lem-acc-F'}, the conjugacy class of $F_{[a,b]}$ accumulates at $G$ in $\sub(G)$. This shows that $G \in \H$, and by minimality $\H =G$.

We now prove that the only URS of Thompson's group $F$ are the normal subgroups of $F$. Starting with a non-trivial $\H \in \urs(F)$ and repeating the exact same argument as above, we obtain the existence of $N \in \H$ such that $[F,F] \leq N$. Such a subgroup has to be normal in $F$, so by minimality we deduce that $\H =N$. This concludes the proof of (i).

\medskip

To prove (ii), observe that the action of $T$ on $\mathbb{S}^1$ is clearly minimal, and is also extremely proximal since it is already the case for the subgroup $[F,F]$ (Lemma \ref{lem-F-ep}). Now for every dyadic numbers $0 < \alpha < \beta < 1$, the rigid stabilizer $T_{[\alpha, \beta]}$ coincides with $F_{[\alpha, \beta]}$, and using Lemma \ref{lem-rigid-stab-F} we see that we are in position to apply Corollary \ref{cor-unique-urs}. Moreover point stabilizers $T_x$, $x \in \mathbb{S}^1$, are maximal subgroups of $T$ since $T$ acts 2-transitively on each of its orbits in $\mathbb{S}^1$ (see \cite[Proposition 1.4]{Sav-F}, where this is proved for the group $F$, but the same proof applies to $T$). Thus (ii) is proved.

The proof of (iii) is very similar to (ii), and actually easier. The action of $V$ on $\Ccal$ is clearly minimal and extremely proximal. Moreover for every cylinder $U$ in $\Ccal$, the rigid stabilizer $V_U$ is isomorphic to $V$, and hence is simple \cite{C-F-P}. Finally for every $x \in \Ccal$, the stabilizer $V_x$ is a maximal subgroup of $V$, because $V$ acts $2$-transitively (indeed $n$-transitively for all $n$) on the orbit of $x$. Therefore we may apply Corollary \ref{cor-unique-urs}, which implies the statement, and concludes the proof of the theorem. \qedhere
\end{proof}

We now give an explicit description of the URS's of the groups $T$ and $V$ arising from their action respectively on the circle and the Cantor set. 

For $x\in \Sbb^1$, we will denote by $T^{0+}_x$ and $T^{0-}_x$ the subgroups of $T$ consisting of those elements that fix pointwise a right (respectively left) neighbourhood of $x$. Note that if $x$ is not a breakpoint of the slope of some element of $T$ (i.e.\ $x$ is not dyadic), then $T^{0+}_x=T^{0-}_x=T^0_x$.

\begin{prop}\label{P: URS T}
The stabilizer URS of $T\acts \Sbb^1$ is given by
\[\S_T(\Sbb^1)=\{T^{0+}_x\, \mid \, x\in D\}\cup \{T^{0-}_x\, \mid \, x\in D\}\cup \{T^{0}_x\, \mid \, x\in\Sbb^1\setminus D\},\]
where $D\subset \Sbb^1$ is the set of dyadic points.
\end{prop}

\begin{proof}
Let $X_0\subset \Sbb^1$ be the domain of continuity of the stabilizer map. Recall that $X_0$ coincides with the set of points $x\in \Sbb^1$ such that $T_x^0=T_x$ (Lemma \ref{lem-cont-stab}), so in particular $X_0$ intersects $D$ trivially. We shall identify the closure in $\sub(T)$ of $\{T^0_x\,\mid\, x\in X_0\}$.

Let $(x_n)$ be a sequence of points of $X_0$ such that $(T^0_{x_n})$ converges to some $H$ in $\sub(T)$. We want to prove that $H$ is either $T^{0+}_x$ or $T^{0+}_x$ for some $x \in D$, or of the form $T^{0}_x$ for some $x \notin D$. Up to passing to a subsequence, we may assume that $(x_n)$ converges to a point $x\in \Sbb^1$, and also that $(x_n)$ converges to $x$ from one side, say from the left. If $(x_n)$ is eventually equal to $x$, then we must have $x \notin D$, and $H = T^0_{x}$. If $(x_n)$ is not eventually constant, then $(T^{0}_{x_n})$ converges to $T^{0-}_x$. To see this, observe that every element of $T^{0-}_x$ belongs to $T^0_{x_n}$ for $n$ large enough, and conversely every element that belongs to infinitely many $T^0_{x_n}$ must belong to $T_x^{0-}$. Thus we have $H=T^{0-}_x$. Since $T^{0-}_x = T^0_x$ when $x$ is not dyadic, we have proved the left-to-right inclusion in the statement. 

Conversely, we shall prove that every $T_x^{0\pm}$, $x\in D$, belongs to the closure of $\{T^0_x\,\mid\, x\in X_0\}$. But this is clear, since if $(x_n)\in X_0$ converges to $x$ from the left (respectively right), then again $(T_{x_n}^0)$ converges to $T_x^{0-}$ (respectively $T_x^{0+}$). Therefore the converse inclusion also holds, and the equality is proved.
\end{proof}

The combination of Theorem \ref{T: URS Thompson} and the above description of the URS associated to $T \acts \Sbb^1$ allows us to deduce the following result.

\begin{thm} \label{thm-T-set-points-stab}
Let $D$ be the set of dyadic points of $\Sbb^1$. The point stabilizers for the action of Thompson's group $T$ on its Furstenberg boundary are either:
\begin{enumerate}[label=(\roman*)]
\item $\{T^{0+}_x\, \mid \, x\in D\}\cup \{T^{0-}_x\, \mid \, x\in D\}\cup \{T^{0}_x\, \mid \, x\in\Sbb^1\setminus D\}$; in which case Thompson's group $F$ is amenable;
\item or trivial; in which case Thompson's group $F$ is non-amenable.
\end{enumerate}
\end{thm}

\begin{proof}
Recall (Proposition \ref{prop-A_G-other-action}) that for every countable group $G$, the set of point stabilizers for the action $G \acts \partial_F G$ is precisely the Furstenberg URS of $G$. Since Thompson's group $T$ is non-amenable, the URS $\mathcal{A}_T$ has to be either trivial or equal to $\S_T(\Sbb^1)$ according to Theorem \ref{T: URS Thompson}. If $\mathcal{A}_T = \S_T(\Sbb^1)$ then statement (i) holds thanks to Proposition \ref{P: URS T}, and in this case the group $F$ must be amenable since all the elements of $\S_T(\Sbb^1)$ contain a copy of $F$. On the other hand if $\mathcal{A}_T$ is trivial then $\S_T(\Sbb^1)$ is not amenable, and therefore the group $F$ is not amenable either because the conjugacy class of $F$ inside $T$ accumulates on $\S_T(\Sbb^1)$.
\end{proof}

For Thompson's group $V$ we have:

\begin{prop} \label{P: URS V}
The stabilizer URS of $V\acts \Ccal$ is given by $\S_V(\Ccal)=\{V^0_x\, \mid\, x\in \Ccal\}$.
\end{prop}

\begin{proof} 
If $x$ is a point of $\Ccal$ and $g$ an element of $V_x$, then either $\fix(g)$ contains a cylinder containing $x$, or there exists a cylinder $U$ around $x$ and $\varepsilon = \pm 1$ such that $g^{\varepsilon n}y \rightarrow x$, $n \rightarrow \infty$, for every $y \in U$. In particular either $\fix(g)$ contains a neighbourhood of $x$, or $x$ is isolated in $\fix(g)$. This implies that $V\acts \Ccal$ has Hausdorff germs (Definition \ref{D: Hausdorff germs}), so the statement follows from Proposition \ref{prop-cont-stab0}. \qedhere 
\end{proof}

\subsubsection{Rigidity of non-free minimal actions}

In this subsection we explain how the classification of the URS's of the Thompson groups, obtained in the previous subsection, imposes strong restrictions on the minimal actions of these groups on compact topological spaces. 

\medskip 

 Recall that given a continuous group action $G\acts X$ on a topological space, for every $x\in X$ we denote $G^0_x$ the subgroup of $G$ consisting of elements that fix pointwise a neigbhourhood of $x$. We will need the following lemma.

\begin{lemma}\label{L: T0} Consider Thompson's group $T$ acting on the circle $\Sbb^1$. For any two distinct points $z_1, z_2\in \Sbb^1$, the subgroups $T^0_{z_1}$ and $T^0_{z_2}$ generate $T$.

The same statement holds for Thompson's group $V$ acting on the binary Cantor set.
\end{lemma}

\begin{proof}
Consider first the case of $T$. Let us first assume that $z_1$ and $z_2$ do not belong to the same $T$-orbit. Then $T_{z_1}^0$ acts transitively on the $T$-orbit of $z_2$ (see the proof of \cite[Proposition 1.4]{Sav-F}). Since for every $g\in T$ we have $gT^0_{z_2}g^{-1}=T^0_{g(z_2)}$, it follows that the subgroup of $T$ generated by $T^0_{z_1}$ and $T^0_{z_2}$ contains the subgroup generated by all conjugates of $T^0_{z_2}$. The latter is a non-trivial normal subgroup of $T$, hence is equal to $T$ since $T$ is simple.

If $z_1$ and $z_2$ belong to the same $T$-orbit, then the group $T_{z_1}^0$ still acts transitively on this orbit after removing the point $z_1$. Hence the same argument shows that the subgroup of $T$ generated by $T^0_{z_1}$ and $T^0_{z_2}$ still contains all conjugates of $T^0_{z_2}$ (by conjugating $T^0_{z_2}$ by elements of $T_{z_1}^0$ we obtain all conjugates but $T^0_{z_1}$ itself, which is already contained in it by assumption). Hence the same argument applies. 

The proof for $V$ is similar (and actually easier). We leave the details to the reader. \qedhere
\end{proof}

Recall that given two group actions $G\acts X$ and $G\acts Y$ by homeomorphisms, we say that $G\acts X$ \textbf{factors onto} $G\acts Y$ if there exists a continuous surjective $G$-equivariant map $X\to Y$. We are now ready to state the following corollary of Theorem \ref{T: URS Thompson}.

\begin{cor} \label{C: rigidity thompson} Let $F$, $T$ and $V$ be the Thompson groups.
\begin{enumerate}[label=(\roman*)]
\item Every faithful, minimal action of $F$ on a compact space is topologically free.
\item Every non-trivial minimal action $T\acts X$ on a compact space which is not topologically free factors onto the standard action on the circle.  
\item Every non-trivial minimal action $V\acts X$ on a compact space which is not topologically free factors onto the standard action on the Cantor set.
\end{enumerate}
Moreover in (ii) and (iii) the factor map $\varphi: X \rightarrow \Sbb^1$ (respectively $\varphi: X \rightarrow \Ccal$) is unique, and is characterized by the condition that $\varphi(x)$ is the unique point of $\Sbb^1$ (respectively $\Ccal$) such that $T^0_{\varphi(x)}$ (respectively $V^0_{\varphi(x)}$) fixes $x$.
\end{cor}

\begin{proof}
Note that (i) follows directly from the first statement of Theorem \ref{T: URS Thompson}, since any minimal action on a compact space giving rise to a trivial URS stabilizer must be topologically free by Proposition \ref{prop-trivial-URS-top-free}.

The proof of (ii) requires additional arguments. Let $T\acts X$ be a non-trivial, minimal action on a compact space, which is not topologically free. The following lemma shows that the map $\varphi$ appearing in the statement is well-defined, and provides a factor map to $T\acts \Sbb^1$. 

\begin{lemma}\label{L: def phi}
For every $x\in X$ there exists a unique $z \in \Sbb^1$ such that $T_{z}^{0} \le T_x$. The map $\varphi\colon X\to \Sbb^1$, $x \mapsto z = \varphi(x)$, is continuous, surjective, and $T$-equivariant.
\end{lemma}

\begin{proof}
Since $T\acts X$ is neither trivial nor topologically free, by Theorem \ref{T: URS Thompson} we must have $\S_T(X)=\S_T(\Sbb^1)$. Hence Lemma \ref{L: upper semicontinuous} ensures the existence of a point $z \in \Sbb^1$ as in the statement. The uniqueness of $z$ follows from the fact that for any two $z_1\neq z_2\in \Sbb^1$, the subgroups $T^{0}_{z_1}$ and $T^{0}_{z_2}$ generate $T$ (Lemma \ref{L: T0}). Hence if there were two possible choices for $z$, the point $x$ would be globally fixed by $T$, contradicting our assumption on $T\acts X$. In the sequel we write $z = \varphi(x)$.


The fact that the map $\varphi$ is equivariant is clear. Let us check that $\varphi$ is continuous. Let $(x_i)_{i\in I}\subset X$ be a net converging to $x$. Let $y\in \Sbb^1$ be a cluster point of $(\varphi(x_i))$ and let us prove that $y=\varphi(x)$.  We may assume, upon taking a subnet, that $\varphi(x_i)\to y$ and that $(T_{x_i})$ converges to a limit $H$ in $\sub(T)$. Since $\varphi(x_i)\to y$, every element of $T^0_y$ eventually belong to $T^0_{\varphi(x_i)}$, and since  $T_{\varphi(x_i)}^0\le T_{x_i}$ we deduce that $T^0_y\le H$. Moreover since $T_{x_i}\to H$ and $x_i\to x$, we deduce that $H\le T_x$. It follows that $T^0_y\le T_x$, and since $\varphi(x)$ is the only point of $\Sbb^1$ with this property, $y=\varphi(x)$.
 
The fact that $\varphi$ is onto readily follows, since the action of $T$ on $\Sbb^1$ is minimal. \qedhere
\end{proof}
Finally note that the map $\varphi$ constructed in the lemma is the unique factor map from $T\acts X$ to $T\acts \Sbb^1$.  Namely if $\psi: X\to \Sbb^1$ is another factor map, by equivariance $T^0_{\varphi(x)}$ must must fix $\psi(x)$, which implies that $\psi(x)=\varphi(x)$.

The proof of (iii) is exactly the same, using the second statement of Lemma \ref{L: T0}. \qedhere
\end{proof}

\subsubsection{Actions of  Thompson's group $T$ on the circle} \label{S: T circle}

In this paragraph we further investigate situation (ii) of Corollary \ref{C: rigidity thompson} for actions of Thompson's group $T$ on the circle.

Recall that a representation $\tau\colon G\to \homeo(\Sbb^1)$ is said to be \textbf{semi-conjugate} to $\sigma\colon G\to\homeo_+( \Sbb^1)$ if $\tau$ factors onto $\sigma$ through a map which is monotone (non-increasing or non-decreasing) with respect to the circular order on $\Sbb^1$, and with degree $\pm 1$.


Ghys and Sergiescu have proved in \cite[Th\'eor\`eme K]{Ghy-Se} that every non-trivial action of Thompson's group $T$ on the circle by $C^2$-diffeomorphisms is semi-conjugate to the standard action. The main purpose of this paragraph is to show that this result actually holds for every action of $T$ on the circle by homeomorphisms.

 
\begin{thm}\label{T: circle}
Every non-trivial, continuous action of Thompson's group $T$ on the circle is semi-conjugate to the standard action.
\end{thm}

First note that a non-trivial continuous action of $T$ on the circle cannot have a finite orbit. Indeed, the case of a finite orbit of cardinality greater than $1$ is ruled out by the simplicity of the group $T$; and the case of a global fixed point cannot happen either because the stabilizer of a point in the group of orientation preserving homeomorphisms of the circle is torsion-free, whereas $T$ contains torsion elements.

It is a well-known general fact that any continuous action of a countable group on the circle without finite orbits is semi-conjugate to a minimal action (the semi-conjugation is obtained by collapsing to a point every connected component of the complement of the unique minimal subset) \cite[Proposition 5.8]{Ghys-circle}. In particular Theorem \ref{T: circle} has the following equivalent formulation.


\begin{thm}\label{T: circle minimal}
Every minimal continuous action of the group $T$ on the circle is conjugate to the standard action.
\end{thm}

\begin{proof}
In the sequel we denote by $\tau\colon T\to\homeo( \Sbb^1)$ an arbitrary representation of $T$ on the circle such that the induced action $\tau(T)\acts \Sbb^1$ is minimal, and by a mere inclusion $ T\hookrightarrow \homeo(\Sbb^1)$ the standard representation. If $x\in \Sbb^1$ and $g\in T$, the notation $gx$ will refer to the standard action, and we will always write $\tau(g)x$ when making reference to the action induced by $\tau$. The notations $T_x, T^0_x$ will always be intended with respect to the standard action. 

We wish to apply Corollary \ref{C: rigidity thompson} to the action $\tau(T)\acts \Sbb^1$. For this, we need to rule out the possibility that this action is topologically free. This will follow from the following well-known lemma.

\begin{lemma} \label{lem-F'-fixed-point-Margulis}
Every non-trivial continuous action of $[F,F]$ on the circle has a fixed point.
\end{lemma}

We derive this lemma from a result of Margulis (conjectured by Ghys), stating that every group of homeomorphisms of the circle that does not have free subgroups must preserve a probability measure \cite{Margulis}. However \'E.\ Ghys informed us that the lemma was known before Margulis' result, and can also be proved directly.

\begin{proof}
Since $[F,F]$ does not have free subgroups by a result of Brin and Squier \cite{Brin-Squier}, Margulis' alternative implies that every continuous action on $\Sbb^1$ preserves a probability measure $\mu$. Assume by contradiction that $[F,F]$ acts on $\Sbb^1$ with no global fixed point. By simplicity of $[F,F]$ all orbits must be infinite, and it follows that the measure $\mu$ is atomless. By reparametrizing the circle using $\mu$, we can semi-conjugate the action of $[F,F]$ to an action that preserves the Lebesgue measure, i.e.\ an action by rotations. This is clearly a contradiction because $[F,F]$, being simple, does not have non-trivial abelian quotients. \qedhere
\end{proof}

\begin{lemma}
A minimal action of $T$ on the circle cannot be topologically free. 
\end{lemma}

\begin{proof}
Let $\tau: T\to \homeo(\Sbb^1)$ be, as above, a representation inducing a minimal action. According to Lemma \ref{lem-F'-fixed-point-Margulis}, there exists $x \in \Sbb^1$ that is fixed by $\tau([F,F])$. Since the conjugacy class of $[F,F]$ in $\sub(T)$ does not accumulate on the trivial subgroup, a fortiori the same is true for the stabilizer of $x$ for the action induced by $\tau$. According to the last sentence in Proposition \ref{prop-min-action-urs}, the stabilizer URS of this action must be non-trivial. By Proposition \ref{prop-trivial-URS-top-free}, this exactly means that the action is not topologically free. \qedhere
\end{proof}


We now complete the proof of Theorem \ref{T: circle minimal}. Since the action $\tau(T)\acts \Sbb^1$ is minimal and not topologically free, Corollary \ref{C: rigidity thompson}(ii) provides us with a continuous surjective map $\varphi: \Sbb^1\to \Sbb^1$ that factors the action $\tau(T)\acts \Sbb^1$ onto the standard action $T\acts \Sbb^1$. Recall the definition of the map $\varphi$: for $x\in \Sbb^1$, $\varphi(x)\in \Sbb^1$ is the unique point $z\in \Sbb^1$ such that $\tau(T^0_{z})$ fixes $x$. Let us show that, in this situation, $\varphi$ must be a homeomorphism. It is enough to check that $\varphi$ is injective. To this end, assume by contradiction that there exist $x_1\neq x_2$ such that $\varphi(x_1)=\varphi(x_2)=z$. Since $\varphi$ is surjective, in particular $\varphi$ is not constant on at least one of the two intervals with endpoints $x_1$ and $x_2$, say $I$. Since $\tau(T_z^0)$ fixes $x_1$ and $x_2$, $\tau(T_z^0)$ must preserve $I$ (note that $\tau(T)$ necessarily acts by orientation-preserving homeomorphisms because $T$ has no subgroup of index two). Choose and fix a point $y\in I$ such that $z'=\varphi(y)$ verifies $T_{z'}^0=T_{z'}$, and $z'$ is not contained in the $T$-orbit of $z$. Note that such a point $y$ exists because $\varphi(I)\subset \Sbb^1$ is a proper interval, and the set of $z'$ verifying these two conditions is dense in $\Sbb^1$ (in fact, both conditions are verified by all but countably many points). The fact that $z$ and $z'$ lie in different $T$-orbits implies that $T^0_z$ acts transitively on the orbit of $z'$, as it was already observed in the proof of Lemma \ref{L: T0}. From this we deduce that every element of $T$ can be written as a product of an element in $T^0_z$ and an element in $T_{z'}=T^0_{z'}$, i.e.\ we have the decomposition $T=T_z^0T^0_{z'}$. Since $\tau(T^0_{z'})$ fixes $y$ by definition of $z'=\varphi(y)$, it follows that the $\tau(T)$-orbit of $y$ is equal to the $\tau(T^0_z)$-orbit of $y$, and thus it is contained in $I$ since $\tau(T^0_z)$ preserves $I$. This contradicts the minimality of $\tau(T)\acts \Sbb^1$, and shows that the map $\varphi$ must be injective. This concludes the proof of Theorem \ref{T: circle minimal}, and thus the proof of Theorem \ref{T: circle}. \qedhere
\end{proof}



The following remark was pointed out to us by \'E.\ Ghys.

\begin{remark}
In contrast to what is shown in Theorem \ref{T: circle} for Thompson's group $T$, the group $F$ does admit non-trivial continuous actions on the interval that are not semi-conjugate to the standard action. Such an action can be obtained by choosing a bi-invariant ordering on $F$ (bi-invariant orderings on $F$ exist, and these were completely classified in \cite{Na-Ri}), embedding $F$ in the interval $[0,1]$ by respecting the bi-invariant ordering, and extending the $F$-action on itself to an action on the interval in the natural way. The resulting action of $F$ on $[0,1]$ has the property that for every $g\in F$, either $gx\leq x$ or $gx\geq x$ for all $x\in [0,1]$. Note that the standard action of $F$ on $[0,1]$ is far from verifying this property.
\end{remark}

\subsection{Groups of piecewise projective homeomorphisms} \label{subsec-projective}

In this paragraph we prove the $C^\ast$-simplicity of the groups of piecewise projective homeomorphisms of the real line considered by Monod in \cite{Monod-Pw} to provide new examples of non-amenable groups without free subgroups. We also prove the $C^\ast$-simplicity of the finitely presented group $G_0$ introduced by Lodha and Moore in \cite{Lod-Moo}.

\medskip

We take the original notation from \cite{Monod-Pw}. Consider the action of $\PSL_2(\R)$ by homographies on the projective line $\mathbb{P}^1(\R)=\R\cup\{\infty\}$. Given a subring $A \leq \R$, we let $G(A)$ be the group of homeomorphimsms of $\mathbb{P}^1(\R)$ that coincide with elements of $\PSL_2(A)$ in restriction to finitely many intervals, and such that the endpoints of the intervals belong to the set of fixed points of hyperbolic elements of $\PSL_2(A)$. Let also $H(A) \leq G(A)$ be the stabilizer of the point $\infty$. Thus, $H(A)$ acts on the real line $\R$ by homeomorphisms.


The following easy lemma provides a sufficient condition for all rigid stabilizers to be non-amenable.

\begin{lemma} \label{lem-orbit-base}
Let $X$ be a topological space, and $G$ a group of homeomorphisms of $X$. Assume that there exists an open subset $U_0 \subset X$ whose $G$-orbit forms a basis for the topology, and such that $G_{U_0}$ is non-amenable. Then all rigid stabilizers $G_U$, for $U \subset X$ open and non-empty, are non-amenable.
\end{lemma}

\begin{proof}
Let $U$ be any non-empty open subset of $X$. If $g \in G$ is such that $g(U_0) \subset U$ (such an element exists by assumption), then $gG_{U_0}g^{-1}$ is contained in $G_U$. Since $G_{U_0}$ is non-amenable by assumption, $G_U$ is also non-amenable.
\end{proof}

\begin{thm}\label{T: projective}
Let $G \leq H(\R)$ be a non-amenable countable subgroup of $H(\R)$. Assume that for every non-empty intervals $I_1 = ]a,b[$ and $I_2 = ]c,d[$, there is $g \in G$ such that $g(I_1) \subset I_2$. Then $G$ is $C^\ast$-simple.
\end{thm}

\begin{proof}
The two germs around the point $\infty$ give rise to a morphism $G \rightarrow (\mathbb{R} \rtimes \mathbb{R}_+^\times)^2$. Since $G$ is non-amenable, the kernel $N$ of the above morphism is also non-amenable.

Since every element of $N$ acts trivially on a neighbourhood of $\infty$, we can write $N$ as the increasing union of the rigid stabilizers $N_{]-n,n[}$, $n \geq 1$. Since a direct limit of amenable groups remains amenable, there must exist $n$ such that $N_{]-n,n[}$ is non-amenable. By assumption the $G$-orbit of $]-n,n[$ generates the topology on $\mathbb{R}$, so by Lemma \ref{lem-orbit-base} all rigid stabilizers $G_U$, for $U \subset \R$ open and non-empty, are non-amenable. Therefore $G$ satisfies the assumptions of Theorem \ref{thm-non-amenab-rigid-stab}, so it follows that $G$ is $C^\ast$-simple.
\end{proof}

It is proved in \cite{Monod-Pw} that if $A \leq \R$ is a dense subring, then the group $H(A)$ is non-amenable. Note that non-amenability of $G(A)$ is clear since $G(A)$ contains non-abelian free subgroups. Therefore Theorem \ref{T: projective} implies:

\begin{cor} \label{cor-H(A)-cstar}
Let $A$ be a countable dense subring of $\mathbb{R}$. Then both $H(A)$ and $G(A)$ are $C^\ast$-simple.
\end{cor}

Lodha and Moore have exhibited a finitely presented subgroup $G_0 \le H(\mathbb{Z}[1/\sqrt{2}])$ \cite{Lod-Moo}. The group $G_0$ is generated by the translation $a = t \mapsto t+1$, and the two following transformations $b$ and $c$:
\[b(t)=\left\{\begin{array}{lr} t & \text{ if }t\leq 0\\
 \frac{t}{1-t} & \text{ if } 0\leq t\leq \frac{1}{2}\\
3-\frac{1}{t} & \text{ if }\frac{1}{2}\leq t \leq 1\\
t+1 & \text{ if }1\leq t\end{array}\right.
\quad
c(t)=\left\{\begin{array}{lr} \frac{2t}{1+t} & \text{ if }0 \leq t\leq 1\\
 t & \text{ otherwise.} \end{array}\right.
\]

It readily follows from the definition that for all $n \geq 1$, the element $b^n$ sends the interval $[0,1]$ to $[0,n]$. Therefore the $\langle a,b\rangle$-orbit of any non-empty open interval generates the topology on $\R$. Combined with the non-amenability of $G_0$ \cite{Lod-Moo}, Theorem \ref{T: projective} thus implies the following result.

\begin{cor}
The group $G_0$ is $C^\ast$-simple.
\end{cor}

As far as we know, this provides the first example of a finitely presented $C^\ast$-simple group with no free subgroups.


\subsection{Groups acting on trees} \label{subsec-trees}

In this paragraph $T$ will be a simplicial tree, whose set of ends will be denoted $\partial T$. Any edge of $T$ separates $T$ into two subtrees, called \textbf{half-trees}.

Recall that $g \in \aut(T)$ is \textbf{elliptic} if $g$ stabilizes a vertex or an edge, and \textbf{hyperbolic} if $g$ translate along a bi-infinite geodesic line, called the axis of $g$. In this last situation $g$ admits exactly two fixed ends, called the \textbf{endpoints} of $g$. We say that the action of a group $G$ on $T$ is \textbf{of general type} if $G$ has hyperbolic elements without common endpoints, and is \textbf{minimal} if $T$ has no proper $G$-invariant subtree. This last terminology is conflicting with the one introduced in \textsection\ref{subsec-notation}, but there will be no ambiguity since the meaning of the word \textit{minimal} will be clear from the context.

\medskip

\textit{In the sequel we assume that $G$ is a subgroup of $\mathrm{Aut}(T)$ whose action on $T$ is minimal and of general type}. We will repeatedly use the classical fact that, under these assumptions, $G$ contains a hyperbolic element with axis contained in any given half-tree of $T$.

\begin{lemma} \label{lem-put-half-tree-inside}
For every half-trees $T_1,T_2$ in $T$, there is $g \in G$ such that $g(T_1) \subset T_2$.
\end{lemma}

\begin{proof}
First assume that the half-trees $T_1$ and $T_2$ are disjoint. Then any hyperbolic $g \in G$ such that the axis of $G$ lies in $T_2$ will send $T_1$ inside $T_2$. If $T_2$ is not contained in $T_1$, then $T_2$ must contain another half-tree $T_2'$ that is disjoint from $T_1$, so that we may apply the previous argument and find $g \in G$ such that $g(T_1) \subset T_2' \subset T_2$. In the case when $T_1$ contains $T_2$, any hyperbolic element whose axis is contained in the complement of $T_1$ will send $T_1$ onto a subtree disjoint from $T_1$, so in particular disjoint from $T_2$. Again we are reduced to the first case, and the statement is proved.  
\end{proof}

We consider the set $T \sqcup \partial T$, endowed with the coarsest topology for which every half-tree is open. By a half-tree in $T \sqcup \partial T$ we mean the union of a half-tree $T_1$ in $T$ and the set of ends defined by $T_1$. The set $T_f$ of vertices of $T$ having finite degree is clearly preserved by $G$, so that $G$ also acts on $X_T = (T \setminus T_f) \sqcup \partial T$. One easily check that $X_T$ is actually the only minimal closed $G$-invariant subset of $T \sqcup \partial T$. 

\begin{prop} \label{prop-tree-ep}
The action of $G$ on $X_T$ is extremely proximal. 
\end{prop}

\begin{proof}
We want to show that every proper closed $C \subset X_T$ is compressible. The subset $\partial T$ being dense in $X_T$, the complement of $C$ must contain some $\xi \in \partial T$. Since any neighbourhood of $\xi$ contains a half-tree, upon enlarging $C$ we may assume that $C$ is itself a half-tree (here by half-tree in $X_T$ we mean the intersection of a half-tree of $T \sqcup \partial T$ with $X_T$). Now given any point $\eta \in \partial T$, Lemma \ref{lem-put-half-tree-inside} shows that any neighbourhood of $\eta$ contains a $G$-translate of $C$, so $C$ is compressible.
\end{proof}

\subsubsection{Almost prescribed local action} \label{subsec-G(F,F')}

In this paragraph $\Omega$ will be a (possibly finite) countable set of cardinality greater than three, and $T$ will be a regular tree of branching degree the cardinality of $\Omega$. Note that the set of vertices of $T$ must be countable. As in the previous paragraph, we denote $X_T = (T \setminus T_f) \sqcup \partial T$. Since the tree $T$ is now regular, we have $X_T = \partial T$ when $\Omega$ is finite, and $X_T = T \sqcup \partial T$ when $\Omega$ is infinite.


Let $c: E(T) \rightarrow \Omega$ be a colouring of the set of edges of $T$ such that for every vertex $v$, the map $c$ induces a bijection $c_v$ between the set of edges around $v$ and $\Omega$. For $g \in \aut(T)$ and for a vertex $v$, the action of $g$ around $v$ gives rise a permutation $\sigma(g,v) \in \mathrm{Sym}(\Omega)$, defined by $\sigma(g,v) = c_{gv} \circ g_v \circ c_v^{-1}$, that we will call the \textbf{local permutation} of $g$ at the vertex $v$.

Given two permutation groups $F \leq F' \leq \mathrm{Sym}(\Omega)$, we denote by $G(F,F')$ the set of $g \in \aut(T)$ having all their local permutations in $F'$, and all but finitely many in $F$: $\sigma(g,v) \in F'$ for all $v$ and $\sigma(g,v) \in F$ for all but finitely many $v$. The index two subgroup of $G(F,F')$ that preserves the types of vertices of $T$ will be denoted $G(F,F')^\ast$. \textit{We will always assume that the permutation group $F$ acts freely on $\Omega$}. Under this assumption, it is not hard to see that $G(F,F')$ is a countable group as soon as the permutation group $F'$ is countable. When $\Omega$ is finite, $G(F,F')$ is actually a finitely generated group \cite{LB-ae}.

\medskip

Recall that, under the assumption that point stabilizers in $F'$ are amenable, the groups $G(F,F')$ are not $C^\ast$-simple and yet do not have non-trivial amenable normal subgroups \cite{LB-c-etoile}. The goal of this paragraph is twofold: we first give an explicit description of the Furstenberg URS of these groups (for $\Omega$ finite) and interpret this result at the level of the Furstenberg boundary; and also give, under appropriate assumptions on the permutation groups, a complete classification of all the URS's of the groups $G(F,F')$.

\subsubsection{Description of the Furstenberg URS}

We will need the following lemma.

\begin{lemma} \label{lem-ell-fix-vois}
Let $\xi \in \partial T$, and denote by $G(F,F')_\xi^0$ the set of elements of $G(F,F')$ fixing a neighbourhood of $\xi$ in $\partial T$. Then $G(F,F')_\xi^0$ is precisely the set of elliptic elements of $G(F,F')_\xi$.
\end{lemma}

\begin{proof}
The fact that $G(F,F')_\xi^0$ contains only elliptic elements is clear. To prove the converse, let $g$ be an elliptic element of $G(F,F')_\xi$, and denote by $(v_n)$ a sequence of adjacent vertices representing the point $\xi$. Since $g$ is elliptic and $g$ fixes $\xi$, $g$ must fix $v_n$ for $n$ large enough. Now since $g$ has only finitely many local permutations outside $F$, there is $n_0 \geq 1$ such that $\sigma(g,v) \in F$ for every vertex $v$ outside the ball of radius $n_0$ around $v_0$. Thanks to the previous observation, we may also assume that g fixes all the $v_n$ for $n \geq n_0$. Now the permutation $\sigma(g,v_{n_0+1})$ belongs to $F$ and has a fixed point since $g$ fixes the edge $e_{n_0}$ between $v_{n_0}$ and $v_{n_0+1}$. Since $F$ acts freely on $\Omega$ by assumption, this shows that $\sigma(g,v_{n_0+1})$ is trivial, i.e.\ $g$ fixes the star around $v_{n_0+1}$. By repeating the argument we immediately see that $g$ must fix pointwise the half-tree defined by $e_{n_0}$ and containing $v_{n_
0+1}$. The latter being an open neighbourhood of $\xi$ in $\partial T$, the statement is proved.
\end{proof}

The following result gives a precise description of the stabilizer URS associated to the action $G(F,F') \acts \partial T$, and shows that it coincides with the maximal amenable URS of $G(F,F')$. We point out that the assumption that $\Omega$ is finite (equivalently, $X_T = \partial T$) is important here, as the description of the URS associated to the action of $G(F,F')$ on $X_T$ turns out to be more complicated when $\Omega$ is infinite.

\begin{prop} \label{prop--urs-tree-G(F,F')}
Assume that $\Omega$ is finite. Let $F \leq F' \leq \mathrm{Sym}(\Omega)$, and write $G = G(F,F')$. Then the Furstenberg URS of $G$ is $\mathcal{A}_G = \mathcal{S}_G(\partial T)$, and is exactly the collection of subgroups $G_\xi^0$, for $\xi \in \partial T$.
\end{prop}

\begin{proof}
The action $G \acts \partial T$ is a boundary action, and stabilizers are (locally finite)-by-cyclic \cite{LB-ae}. In particular they are amenable, so the equality $\mathcal{A}_G = \mathcal{S}_G(\partial T)$ follows from Proposition \ref{prop-A_G-other-action}. To show the second statement, take $\xi \in \partial T$ and $g \in G_\xi$. If $g$ is hyperbolic then $\xi$ is clearly isolated in $\fix(g)$, and if $g$ is elliptic then $\fix(g)$ contains a neighbourhood of $\xi$ according to Lemma \ref{lem-ell-fix-vois}. Therefore the action of $G$ on $X_T = \partial T$ has Hausdorff germs, and the conclusion follows from Proposition \ref{prop-cont-stab0}.
\end{proof}


\begin{cor} \label{cor-fix-G(F,F')-furst}
Assume that $\Omega$ is finite, and let $F \leq F' \leq \mathrm{Sym}(\Omega)$.
\begin{enumerate}[label=(\roman*)]
\item The collection of point stabilizers for the action of $G(F,F')$ on its Furstenberg boundary is exactly the collection of $G(F,F')_\xi^0$, $\xi \in \partial T$.
\item An element $g \in G(F,F')$ fixes a point in the Furstenberg boundary of $G(F,F')$ if and only if $g$ fixes a half-tree of $T$, or equivalently $g$ is elliptic and $g$ fixes a point in $\partial T$.
\end{enumerate}
\end{cor}

\begin{proof}
The elements of $\mathcal{A}_{G(F,F')}$ are exactly the points stabilizers for the action of $G(F,F')$ on its Furstenberg boundary (Proposition \ref{prop-A_G-other-action}), so the first statement is a consequence of Proposition \ref{prop--urs-tree-G(F,F')}. The second statement follows immediately thanks to Lemma \ref{lem-ell-fix-vois}.
\end{proof}

The following remark shows that the stabilizer map $\A_G \rightarrow \sub(G)$, which associates to $H \in \A_G$ its normalizer, is not continuous in general.

\begin{remark}[see also Remark \ref{rem-AG}(3)] \label{rmq-normalizer-not-cont}
Assume that $F \lneq F'$, and write $G = G(F,F')$. Then the map $\sub(G) \rightarrow \sub(G)$, which sends a subgroup $H$ to its normalizer $N_H$ in $G$, is not continuous on $\mathcal{A}_{G}$. Indeed, let $\xi \in \partial T$ be the endpoint of some hyperbolic element of $G$, and choose a sequence $(\xi_n)$ converging to $\xi$ such that $\xi_n$ is not the endpoint of some hyperbolic element of $G$. Then $G_{\xi_n}^0$ converges to $G_\xi^0$ in $\mathcal{A}_{G}$, but $N_{G_{\xi_n}^0} = G_{\xi_n}^0$ does not converge to $N_{G_{\xi}^0}$ because the latter contains a hyperbolic element by our assumption on $\xi$.
\end{remark}

\subsubsection{Classification of all URS's}

Sufficient conditions on the permutation groups $F \leq F'$ ensuring the simplicity of the group $G(F,F')^\ast$ were obtained in \cite[Corollary 4.14]{LB-ae}. In this paragraph we strengthen this result by giving sufficient conditions under which we are able to completely describe the set of URS's of $G(F,F')^\ast$ (see Theorem \ref{thm-all-urs-g(f,f')}). Note that these conditions are nevertheless strictly stronger than the ones from \cite[Corollary 4.14]{LB-ae}. 


\medskip

We fix a vertex $o \in T$. To every $\xi \in \partial T$ is associated a Busemann (or height) function $b_\xi : T \rightarrow \mathbb{Z}$, defined by $b_\xi(v) = d(v,\hat{ov}) - d(o,\hat{ov})$, where $\hat{ov}$ is the projection of the vertex $v$ on the geodesic ray $[o,\xi[$. For every $k \in \mathbb{Z}$, we denote by $\mathcal{L}_{\xi,k}$ the level set $b_\xi^{-1}(k)$. These level sets $\mathcal{L}_{\xi,k}$ partition the set of vertices of $T$, and every vertex $v \in \mathcal{L}_{\xi,k}$ admits exactly one neighbour in $\mathcal{L}_{\xi,k-1}$, that will be denoted $v_-$. 

\begin{lemma} \label{lem-trans-horospheres}
Assume that $F'$ acts $2$-transitively on $\Omega$, and fix $\xi \in \partial T$. Then the group $G(F,F')_\xi$ acts transitively on $\mathcal{L}_{\xi,k}$ for every $k \in \mathbb{Z}$. 
\end{lemma}

\begin{proof}
First note that two vertices in $\mathcal{L}_{\xi,k}$ must be at even distance from each other. We let $v,w \in \mathcal{L}_{\xi,k}$, and we prove that $v$ and $w$ are in the same $G(F,F')_\xi$-orbit. We argue by induction on $d(v,w) = 2n$. 

The case $n = 0$ is trivial. Assume that $n \geq 1$. Let $m$ be the midpoint of the geodesic $[v,w]$, which is also the unique vertex of $[v,w]$ that belongs to $\mathcal{L}_{\xi,k-n}$. Denote by $a \in \Omega$ the color of the edge $(m,m_-)$, and by $x$ (resp.\ $y$) the neighbour of $m$ that belongs to $[m,v]$ (resp.\ $[m,w]$). Since $F_a'$ (the stabilizer of $a$ in $F'$) acts transitively on $\Omega \setminus \left\{a\right\}$, we may find $g \in G(F,F')$ such that $g$ fixes pointwise the half-tree defined by the edge $(m,m_-)$ and containing $m_-$, and $g(x) = y$ (see for instance Lemma 3.4 in \cite{LB-ae}). Clearly such an element $g$ belongs to $G(F,F')_\xi$. Now by construction the vertices $g(v)$ and $w$ are at distance at most $2n-2$ from each other, and the conclusion follows by induction. 
\end{proof}

\begin{prop} \label{prop-stab-end-max}
Assume that $F'$ acts $2$-transitively on $\Omega$, and fix $\xi \in \partial T$. Then every subgroup $H \leq G(F,F')$ strictly containing $G(F,F')_{\xi}$ is either $G(F,F')^\ast$ or $G(F,F')$.
\end{prop}

\begin{proof}
Since $H$ contains $G(F,F')_{\xi}$ as a proper subgroup, we may find $h \in H$ and vertices $v,w$ such that $h(v) = w$ and $h(v_-) \neq w_-$. We claim that this implies that $H_v$, the stabilizer of $v$ in $H$, acts transitively on the star around $v$. To see this, first remark that all the neighbours of $v$ different from $v_-$ are in the same $H_v$-orbit by Lemma \ref{lem-trans-horospheres}. So proving that $H_v$ does not fix $v_-$ is enough to prove the claim. But this is true, because if $g \in G(F,F')_\xi$ fixes $w$ and satisfies $g(h(v_-)) \neq h(v_-)$ (such an element exists because $h(v_-) \neq w_-$), then $h^{-1}gh$ belongs to $H$ and does not fix $v_-$. So we have proved that the stabilizer of $v$ in $H$ acts transitively on the star around $v$.

Now combining this with the fact that $H$ acts transitively on each $\mathcal{L}_{\xi,k}$ thanks to Lemma \ref{lem-trans-horospheres}, we easily deduce that the stabilizer of $v$ in $H$ acts transitively on the star around $v$ for every vertex $v$. This implies in particular that the action of $H$ on $T$ has two orbits of vertices and is of general type. Since moreover $H$ contains the pointwise fixator of any half-tree containing $\xi$ (because $G(F,F')_\xi \leq H)$, we deduce from Lemma \ref{lem-put-half-tree-inside} that $H$ contains all fixators of half-trees. Now the subgroup generated by the pointwise fixators of half-trees in the same as the subgroup generated by pointwise fixators of edges, because the group $G(F,F')$ has the edge-independence property (see \cite{LB-ae}). Finally since $F'$ is $2$-transitive (which implies that $F'$ is generated by its point stabilizers), this subgroup is equal to $G(F,F')^\ast$ according to \cite[Proposition 4.7]{LB-ae}. Therefore $H$ contains $G(F,F')^\ast$, and 
the statement is proved.
\end{proof}


Before stating the main result of this paragraph, let us mention that the $2$-transitivity assumption on the permutation group $F'$ is quite natural. As illustrated by the work of Burger and Mozes \cite{BM-IHES}, the properties of the local action of groups acting on trees is inherent to the structure of these groups, and the $2$-transitivity condition on the local action naturally appears in this setting \cite{BM-IHES}.

Recall that $\Omega$ is a (possibly finite) countable set, and that $X_T$ is either $\partial T$ when $\Omega$ is finite, or $T \sqcup \partial T$ when $\Omega$ is infinite. Examples of finite permutation groups satisfying the assumptions of the following theorem are $F = \left\langle (1,\ldots,d)\right\rangle$ and $F' = \mathrm{Alt}(d)$ for $d \geq 7$ odd. More examples may be found in \cite[Example 3.3.1]{BM-IHES}.

\begin{thm} \label{thm-all-urs-g(f,f')}
Let $F \leq F' \leq \mathrm{Sym}(\Omega)$ such that $F$ acts freely transitively on $\Omega$, $F'$ acts $2$-transitively on $\Omega$, and point stabilizers in $F'$ are perfect. Then the group $G(F,F')^\ast$ admits exactly three URS's, namely $1$, $\mathcal{S}_{G(F,F')^\ast}(X_T)$ and $G(F,F')^\ast$. 
\end{thm}

\begin{proof}
We claim that Corollary \ref{cor-unique-urs} applies to the action of $G(F,F')^\ast$ on $X_T$. This action is minimal and extremely proximal by Proposition \ref{prop-tree-ep}. Since point stabilizers in $F'$ are perfect, fixators of half-trees in $G(F,F')^\ast$ are generated by perfect subgroups, and hence are perfect. By combining this observation with Lemma \ref{double-comm}, we see that every finite index subgroup in the fixator of a half-tree must contain the fixator of another half-tree, so that the second assumption of Corollary \ref{cor-unique-urs} is also satisfied. Finally since stabilizers of ends are maximal subgroups of $G(F,F')^\ast$ according to Proposition \ref{prop-stab-end-max}, the conclusion follows from Corollary \ref{cor-unique-urs}.
\end{proof}

\subsubsection{Piecewise prescribed tree automorphisms}

In this paragraph we consider the \enquote{piecewise-ation} process for groups acting on trees introduced in \cite{LB-c-etoile}. Recall that the construction takes as input a group $G \leq \autT$, and produces a larger group $\PwG$ consisting of automorphisms acting on $T$ piecewise like $G$. See below for a precise definition. While this construction was used in \cite{LB-c-etoile} to produce examples of non $C^\ast$-simple groups with trivial amenable radical, here in contrast we apply the results of Section \ref{sec-main-thms} to show that, under different assumptions on the group $G$ we start with, we obtain a group $\PwG$ that \textit{is} $C^\ast$-simple.

\medskip

Let us first recall the definition of this construction. If $A$ is a finite subtree of $T$ and $v_1, \ldots, v_n$ are the vertices of $A$ having a neighbour outside $T$, we denote by $T_i$ the subtree of $T$ whose projection on $A$ is $v_i$, and by $T \setminus A$ the disjoint union of the subtrees $T_i$. If $G$ is a subgroup of $\autT$, we denote by $\PwG \leq \autT$ the group of automorphisms of $T$ acting piecewise like $G$: an element $\gamma \in \autT$ belongs to $\PwG$ if and only if there exists a finite subtree $A$ such that, if we denote $T \setminus A = \sqcup_{i=1}^n T_i$, then for every $i$ there is $g_i \in G$ such that $\gamma$ and $g_i$ coincide on the subtree $T_i$. 

\medskip

Recall that it was proved in \cite{LB-c-etoile} that, under the assumption that vertex stabilizers in $G$ are amenable (and non-trivial), the group $\PwG$ is \textit{not} $C^\ast$-simple. The following result essentially proves the converse, thereby providing a fairly complete picture of $C^\ast$-simplicity for this class of groups.
 
\begin{thm} \label{thm-free-PwG}
Let $G$ be a countable subgroup of $\autT$ whose action on $T$ is minimal and of general type, and such that fixators of edges in $G$ are non-amenable. Then the group $\PwG$ is $C^\ast$-simple.
\end{thm}

\begin{proof}
According to Theorem \ref{thm-non-amenab-rigid-stab} applied to the action of $\PwG$ on $\partial T$ (where $\partial T$ is endowed with the topology inherited from $X_T$), in order to obtain the conclusion it is enough to show that fixators of half-trees in $\PwG$ are non-amenable. We remark that by Lemma \ref{lem-put-half-tree-inside}, it is actually enough to show that \textit{some} fixator of half-tree in $\PwG$ is non-amenable.

Let $e$ be an edge of $T$, and denote by $T_1$ and $T_2$ the two half-trees separated by $e$. The action of the fixator of $e$ in $G$ on each half-tree induces an embedding $i: G_e \rightarrow \mathrm{Aut}(T_1) \times \mathrm{Aut}(T_2)$. By assumption $G_e$ is non-amenable, so there must exist one of the two half-trees, say $T_1$, such that the image of $p_1 \circ i$, where $p_1$ is the projection of $\mathrm{Aut}(T_1) \times \mathrm{Aut}(T_2)$ on $\mathrm{Aut}(T_1)$, has non-amenable image. 

For every $g \in G_e$, consider the automorphism $\gamma_g$ of $T$ fixing $e$, acting like $g$ on $T_1$ and being the identity on $T_2$. By construction $\gamma_g \in \PwG$, and the map $G_e \rightarrow \PwG$, defined by $g \mapsto \gamma_g$, is a group homomorphism with non-amenable image. This shows that the fixator of $T_2$ in $\PwG$ is non-amenable, and concludes the proof.
\end{proof}

The following example may be compared with Corollary \ref{cor-H(A)-cstar}.

\begin{cor}
The group of automorphisms of the tree $T_{p+1}$ acting piecewise like $\mathrm{PSL}(2,\mathbb{Z}[1/p])$ is $C^\ast$-simple.
\end{cor}

\begin{proof}
Indeed, the action of $\mathrm{PSL}(2,\mathbb{Z}[1/p])$ has two orbits of vertices and does not fix any end of $T_{p+1}$ (the Bruhat-Tits tree of $\mathrm{PSL}(2,\mathbb{Q}_p)$). Moreover edge stabilizers contain free subgroups (as it is already the case for $\mathrm{PSL}(2,\mathbb{Z})$), so $C^\ast$-simplicity follows from Theorem \ref{thm-free-PwG}.
\end{proof}
\subsection{Branch groups} \label{subsec-branch}

We briefly recall basic definitions on groups acting on rooted trees and branch groups. We refer the reader to \cite{B-G-S-branch} for a comprehensive survey on branch groups. 

A rooted tree is said to be \textbf{spherically homogeneous} if vertices at the same distance from the root have the same degree. We denote by $r=(r_1, r_2, \ldots)$ a sequence of integers $r_i \geq 2$, and by $T_r$ a spherically homogeneous rooted tree with degree sequence $r$. The set of vertices at distance $n$ from the root is called the \textbf{$n$-th level} of the tree. The distance between a vertex $v\in T_r$ and the root will be denoted $|v|$. The \textbf{subtree below} $v$ is the tree spanned by all vertices $w$ such that the unique geodesic from $w$ to the root passes through $v$.

We denote $\aut(T_r)$ the group of automorphisms of $T_r$ that fix the root. Note that $\aut(T_r)$ preserves the levels of the tree. A subgroup $G \leq \aut(T_r)$ is \textbf{level-transitive} if $G$ acts transitively on all levels of the tree. All the subgroups of $\aut(T_r)$ that we will consider will be level-transitive. The stabilizer of a vertex $v\in T_r$ in $G$ is denoted $\st_G(v)$. The pointwise stabilizer of the $n$-th level is denoted $\st_G(n)$. The subgroup of $\st_G(v)$ consisting of those elements that act trivially outside the subtree below $v$ is called the \textbf{rigid stabilizer} of $v$ in $G$, and is denoted $\rist_G(v)$. Note that this terminology is consistent with the use of \enquote{rigid stabilizer} elsewhere in the paper, in the sense that if we view $G$ as a group of homeomorphisms of the boundary $\partial T_r$, then $\rist_G(v)$ is precisely the rigid stabilizer of the set of ends defined by the subtree below $v$. The terminology rigid stabilizer for general group actions on topological spaces is actually inspired by the well-established use of this terminology in the world of branch groups. The subgroup of $\st_G(n)$ generated by all rigid stabilizers of vertices at the $n$-th level is called the \textbf{$n$-th level rigid stabilizer}, and is denoted $\rist_G(n)$. It follows from the definition that $\rist_G(n)$ is naturally isomorphic to the direct product $\prod_{|v|=n}\rist_G(v)$. 

A subgroup $G<\aut(T_r)$ is a \textbf{branch group} if $G$ is level-transitive, and the $n$-th level rigid stabilizer $\rist_G(n)$ has finite index in $\st_G(n)$ for every $n \geq 1$.

Many well-studied branch groups are amenable (such as Grigorchuk groups \cite{Grigorchuk:grigorchukgroups} and Gupta-Sidki groups \cite{Gupta-Sidki}). However the branch property does not imply amenability. Sidki and Wilson have constructed examples of branch groups containing free subgroups \cite{S-W-branch}. 

The following theorem shows that amenability is the only obstruction to $C^\ast$-simplicity in the class of branch groups.

\begin{thm} \label{thm-dichotomie-branch}
A countable branch group is either amenable or $C^\ast$-simple.
\end{thm}

According to Theorem \ref{thm-cstar-chab}, part (i) of the following result will imply Theorem \ref{thm-dichotomie-branch}.

\begin{prop} \label{prop-properties-branch}
Let $G$ be a countable branch group, and let $\H \in \urs(G)$ be a non-trivial uniformly recurrent subgroup of $G$. Then:
\begin{enumerate}[label=(\roman*)]
\item If $G$ is not amenable, then neither is $\H$.
\item If $G$ admits non-abelian free subgroups, then so does $\H$.
\item If $G$ is finitely generated, then $\H$ is not elementary amenable.
\end{enumerate}
\end{prop}




\begin{proof}
We view $G$ as a group of homeomorphisms of the boundary of the tree $\partial T_r$, and we apply Corollary \ref{cor-mostgeneral-Q}. Since cylinder subsets form a basis for the topology, we deduce that, in each of the three cases, it is enough to show that all rigid stabilizers $\operatorname{RiSt}(w)$ have the desired property.

In order to prove (ii), assume that $G$ admits non-abelian free subgroups, and fix $n \geq 1$. Since $\rist_G(n)$ has finite index in $G$, $\rist_G(n)$ also admits non-abelian free subgroups. It is not hard to see (see for instance \cite[Lemma 3.2]{Nek10}) that this implies that there must exist a vertex $w$ of level $n$ such that $\operatorname{RiSt}(w)$ contains a free subgroup. Now since rigid stabilizers corresponding to vertices of the same level are conjugated in $G$ (since $G$ is level-transitive), it follows that $\operatorname{RiSt}(v)$ contains a free subgroup for every vertex $v$ of level $n$. According to the previous paragraph, this proves (ii). The proof of (i) is similar. 



Assume now that $G$ is finitely generated. By the main result of \cite{Ju-EG}, $G$ cannot be elementary amenable. Arguing as above, we see that all the subgroups $\rist_G(w)$ are not elementary amenable. Again, according to the first paragraph, this implies that $\H$ is not elementary amenable . \qedhere
\end{proof}

\subsection{Topological full groups}

Let $\Gamma\acts X$ be a group acting by homeomorphisms on a topological space $X$. Recall that the \textbf{topological full group} of the action is the group of all homeomorphisms of $X$ that locally coincide with elements of $\Gamma$.
\begin{thm}\label{P: top full}
Let $\Gamma$ be a countable non-amenable group and $\Gamma\acts X$ be a free, minimal action of $\Gamma$ on the Cantor set. Then the topological full group of $\Gamma \acts X$ is $C^*$-simple.
\end{thm}
We first first recall the following classical consequence of minimality, see \cite{Got}.

\begin{lemma} \label{L: minimality}
Let $\Gamma$ be a countable group and $\Gamma \acts X$ be a minimal action of $\Gamma$ on a compact space. Let $U\subset X$ be an open set. Then there exists  a finite subset $T\subset \Gamma$ such that for every $x\in X$ and every $g\in \Gamma$, there exists $h\in Tg$ such that $hx\in U$.
\end{lemma}
\begin{proof}[Proof of Theorem \ref{P: top full}]
Denote $G$ the topological full group.  By Theorem \ref{thm-non-amenab-rigid-stab}, it is sufficient to show that the rigid stabilizers $G_U$ are non-amenable for every non-empty clopen $U\subset X$. To do this, we use a \enquote{graphing} argument: we construct a finitely generated subgroup of $G_U$ that acts on $U$ with non-amenable orbital Schreier graphs.

Let $S\subset \Gamma$ be a finite subset given by  Lemma \ref{L: minimality} applied to $U$. Up to enlarging $S$ if necessary, we may assume that it is symmetric, contains $1$, and that it generates a non-amenable subgroup of $\Gamma$. From now on, let $\Gamma_0=\langle S\rangle$.

Consider the corresponding Cayley graph $\cay(\Gamma_0, S)$. Choose and fix $x\in U$. Consider

\[\Delta_x=\{ \gamma \in \Gamma_0\, \mid\, \gamma x\in U\}\subset \cay(\Gamma_0, S).\]

Since $x$ has trivial stabilizer in $\Gamma_0$, the set $\Delta_x$ is naturally in bijection with the intersection of the $\Gamma_0$ orbit of $x$ and $U$.

By Lemma \ref{L: minimality} and the choices made, $\Delta_x$ is a $1$-dense subset of $\cay(\Gamma_0, S)$ (meaning that every point of $\cay(\Gamma, S)$ is either in $\Delta_x$ or has a neighbour in $\Delta_x$). Endow $\Delta_x$ with a graph structure by connecting two points if and only if they are at distance at most $3$ in $\cay(\Gamma_0, S)$. 

\begin{lemma}
Endowed with this graph structure and the corresponding metric, $\Delta_x$ is quasi-isometric to $\cay(\Gamma_0, S)$. 
\end{lemma}

\begin{proof}
Let $\delta$ be the distance on $\Delta_x$ and $d$ be the distance on $\Gamma_0$. We already know that $\Delta_x$ is $1$-dense. Clearly if $y, z\in \Delta_x$ we have $d(y,z)\leq3\delta(y,z)$. Let us check that $\delta(y,z)\leq d(y, z)$. Indeed let $y_0=y, y_1,\ldots y_n=z$ be a geodesic in $\Gamma_0$ between $y, z$, with $n=d(y, z)$. For every $i$ there exists $w_i\in \Delta_x$ such that $d(w_i, y_i)\leq 1$, with $w_0=y, w_n=z$. By the triangle inequality $d(w_i,w_{i+1})\leq d(w_i, y_i)+d(y_i, y_{i+1})+d(y_{i+1}, w_{i+1})\leq 3 $ and hence $w_i$ and $w_{i+1}$ are neighbours in $\Delta_x$. Hence $\delta(y, z)\leq n=d(x,y)$. \qedhere

\end{proof}

Set $T=S^3\setminus\{1\}$. For $t\in T$ we denote $U_t=U\cap t^{-1}(U)$. 

\begin{lemma}
For every $t\in T$ there exists a finite partition $\P_t$ of $U_t$ into clopen sets such that for every $V\in \P_t$, we have $t(V)\cap V=\varnothing$ and $t(V)\subset U$.
\end{lemma}

\begin{proof}
Since the action is free, we have $t(x)\neq x$ for every $x\in U_t$, and moreover $t(x)\in U$ by definition of $U_t$. Hence we may cover $U_t$ with finitely many clopen sets verifying the conclusion. After refining this cover we may assume that it is a partition.\qedhere
\end{proof}

For every $t\in T$ and $V\in \P_t$ we define the following element $\gamma_{t, V}$ of the topological full group.
\[\gamma_{t, V}(x)=\left\{\begin{array}{lr} t(x) &  \text{ if }x\in V\\
t^{-1}(x) & \text{ if }x\in t(V)\\
x & \text{otherwise.}
\end{array}\right.\]
Clearly $\gamma_{t, V}\in G_U$ for every $t\in T$. Consider the subgroup $H \leq G_U$ generated by elements $\gamma_{t, V}$ when $t$ ranges in $T$ and $V$ ranges in $\P_t$. 
Observe that by construction the $H$-orbits of every $x\in U$ coincides with the $\Gamma_0$-orbit of $x$ intersected with $U$, hence it is identified with the vertex set of $\Delta_x$. Moreover the orbital Schreier graph of $H$ acting on the orbit of $X$ with respect to the generating set $\{\gamma_{t, V}\, \mid \, t\in T, V\in \P_t\}$ coincides with the previously defined graph structure on $\Delta_x$ (possibly after adding loops and multiple edges). Since this graph is quasi-isometric to $\cay(\Gamma_0, S)$, it is non-amenable. It follows that $H$ has a non-amenable Schreier graph, and thus $H$ is non-amenable. A fortiori the same is true for $G_U$. By Theorem \ref{thm-non-amenab-rigid-stab}, this concludes the proof of the theorem. \qedhere
\end{proof}

\begin{remark}
The assumption that the action $\Gamma\acts X$ is minimal cannot be removed, as the following example shows. Start with any minimal, free action $\Gamma \acts X$ of a non-amenable group on the Cantor set. Let $\hat{\Gamma}=\Gamma\cup\{\infty\}$ be the one-point compactification of $\Gamma$, on which $\Gamma$ acts by fixing the point at infinity. Let $Y=\hat{\Gamma}\times X$, which is still homeomorphic to the Cantor set. Consider the diagonal action $\Gamma\acts Y$, and let $G$ be its topological full group. Observe that the action $\Gamma \acts Y$ is free, but not minimal: in fact, the subset $\{\infty\}\times X$ is a closed minimal invariant subset (and is the unique such subset). Let $N\unlhd G$ be the subgroup consisting of elements that act trivially on $\{\infty\}\times X$. It is clearly non-trivial and normal, as it is the kernel of the restriction of the action of $G$ on $\{\infty\}\times X$. Moreover one can prove that $N$ is locally finite. It follows that $G$ has a non-trivial amenable normal 
subgroup, so it follows that $G$ is not $C^\ast$-simple.
\end{remark}

We do not know, however, if the assumption that the action is \emph{free} can be relaxed to \emph{faithful}.

\bibliographystyle{amsalpha}
\bibliography{subgroupdynamics}

\end{document}